\documentclass[reqno,a4paper]{amsart}

\usepackage{hyperref,amsrefs} 
\usepackage[latin1]{inputenc}
\usepackage{amssymb}
\usepackage{mathrsfs}
\usepackage{graphicx}
\newtheorem{theorem}{Theorem}
\newtheorem{lemma}{Lemma}
\newtheorem{prop}{Proposition}
\newtheorem{cor}{Corollary}
\theoremstyle{definition}
\newtheorem{definition}{Definition}
\newtheorem{remark}{Remark}
\newtheorem{example}{Example}
\newtheorem{cond}{Condition}
\DeclareMathOperator\Prim{Prim}
\DeclareMathOperator\End{End}

\begin{document}

\author{A.~Baldare}
\address{Alexandre Baldare, Institute of Analysis,
Leibniz University Hannover,
Welfengarten 1,
30167 Hannover, Germany}
\email{alexandre.baldare@ac-montpellier.fr}

\author{V.~E.~Nazaikinskii}
\address{Vladimir~E.~Nazaikinskii, Ishlinsky Institute for Problems in Mechanics RAS, Moscow, 119526 Russia}
\email{nazaikinskii@yandex.ru}

\author{A.~Yu.~Savin}
\address{Anton Yu. Savin,  
Peoples' Friendship University of Russia (RUDN University),
6 Miklukho-Maklaya St
117198 Moscow, Russia
}
\email{a.yu.savin@gmail.com}

\author{E.~Schrohe}
\address{Elmar Schrohe, Institute of Analysis,
Leibniz University Hannover,
\mbox{Welfengarten 1},
30167 Hannover, Germany}
\email{schrohe@math.uni-hannover.de}

\title[$C^*$-algebras of elliptic transmission problems with shift operators]{$C^*$-algebras of transmission problems and  elliptic boundary value problems with shift operators}

\begin{abstract}
We study the Fredholm solvability for a new class of nonlocal boundary value problems associated with group actions on smooth manifolds. Namely, we consider the case in which the group action is defined on an ambient manifold without boundary and does not preserve  the manifold with boundary on which the problem is stated. In particular, the group action does not map the boundary to itself. The orbits of the boundary under the group action split the manifold into subdomains, and this decomposition, being combined with the $C^*$-algebra techniques, plays an important role in our approach to the analysis of the problem.
\end{abstract}

\markright{$C^*$-algebras of elliptic transmission problems with shift operators}
\setcounter{footnote}{1}
\footnotetext{The reported study was funded by RFBR,  project 21-51-12006, and DFG, project SCHR 319/10-1.}

\markright{$C^*$-algebras of elliptic transmission problems with shift operators}

\maketitle

\section{Introduction}
The aim of this paper is to study the Fredholm solvability for a new class of nonlocal boundary value problems associated with group actions on smooth manifolds.
Unlike the classical elliptic theory, where the ellipticity condition is stated at each point of the cosphere bundle of the manifold,  the Fredholm property of nonlocal problems associated with group actions requires ellipticity conditions stated on the orbits of these points. This is a consequence of the fact that the group action induces representations in the associated function spaces  by shift operators, which are nonlocal.
Of particular importance and difficulty is the case, when the group and its orbits are infinite. Here, methods of $C^*$-algebras can be applied  successfully (see~\cite{Ant1,Ant2,AnLe1, AnLe2, AnLe3} and the references therein).
In the cited works, one considers a compact smooth manifold with boundary and a discrete group acting on this manifold (or only on its boundary) by diffeomorphisms. In particular, the boundary is invariant under the group action. Such nonlocal problems are studied as follows. First, one reduces the boundary value problem to a zero-order pseudodifferential  boundary problem. Then the algebra of pseudodifferential boundary value problems of order zero is completed to a $C^*$-algebra, its algebra of symbols is identified and its spectrum is computed. Finally, one can use general results from $C^*$-algebras (namely the isomorphism theorem in~\cite{AnLe2}) to define symbols for nonlocal boundary value problems and prove the Fredholm property.

Note, however, that numerous problems in applications (e.g., in mechanics and  probability, see~\cite{BiSo3,OnSk1,OnTs1,OnSk2,Sku1, Sku2, Sku3} and the references therein) lead to a different class of nonlocal problems, for which the group action does not preserve  the manifold with boundary in which the problem is stated (see examples in Fig.~\ref{fig1}). In this case one has to consider the action of a group $\Gamma$ on an ambient manifold without boundary and a submanifold $M$ with boundary that is not invariant under the group action. In particular, the group action does not map $\partial M$ to itself. Then the orbits of the boundary under the group action decompose $M$ into subdomains and this decomposition plays an important role in the analysis of the problem (see~\cite{Sku1}). Note that this decomposition is finite and each subdomain has smooth boundary in Fig.~\ref{fig1}a; is finite but the boundaries of the subdomains have singularities in Fig.~\ref{fig1}b; is infinite  but subdomains have smooth boundaries in Fig.~\ref{fig1}c. It turns out that these three cases are fundamentally different and require quite different methods to study them (we refer to~\cite{Sku1,SkTs1,Ross1} for details).

\begin{figure}
\label{fig1}
\centering
\includegraphics[width=0.9\textwidth]{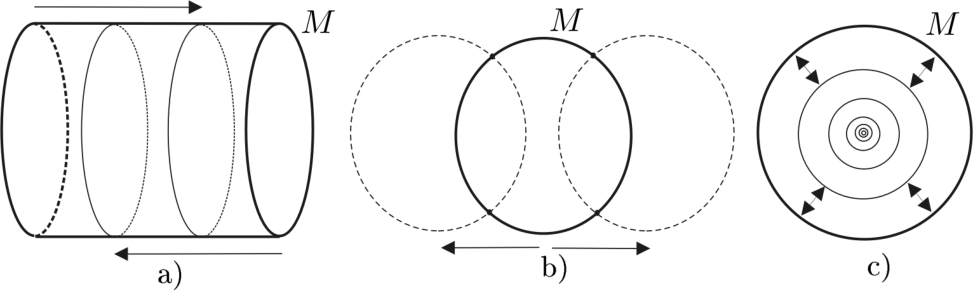}
\caption{a) shifts of finite cylinder; b) shifts of a disc; c) dilations of disc.}
\end{figure}

Somewhat surprisingly, these efficient  $C^*$-algebraic methods  have not yet been applied to such problems. The aim of this paper is to bridge this gap. We combine both techniques from $C^*$-algebras and the decomposition of our manifold with boundary into subdomains. We work under the regularity assumption that two images of the boundary under diffeomorphisms in the group are either disjoint or coincide. In particular, our theorems apply to the situations described in Fig.~\ref{fig1}, a) and c) but do  not apply to b). This regularity condition implies that the union of the images of the boundary is a smooth manifold (possibly noncompact) and we can use the pseudodifferential calculus of boundary value problems with transmission conditions (see~\cite[Section 3.1]{ReSc1}). We also tackle the  technically difficult case, when the images of the boundary $\gamma(\partial M)$, $\gamma\in\Gamma$, under the group action  accumulate in $M$ as in Fig.~\ref{fig1}c.

Let us briefly describe the contents of the paper. We start in Section~\ref{sec2} by recalling the main definitions and results  about Boutet de Monvel transmission problems from~\cite[Section 3.1]{ReSc1}. These are matrix problems  for pseudodifferential operators on a manifold $Y$  with conditions on a closed submanifold $Z$ of codimension one. Such problems form an algebra and we describe their symbols. In Section~\ref{sec3},
we assume that we are given an action of a discrete group $\Gamma$ on $Y$ by diffeomorphisms and that this action satisfies the above regularity assumption with respect to $Z$. Then the  set $X=\Gamma(Z)$ of orbits of $Z$ is in general a noncompact manifold but not a submanifold in $Y$ due to the possible accumulation of the  sets $\gamma(Z)$, $\gamma \in \Gamma$. The algebra of transmission problems on $Y$ with transmission conditions on $X$ is denoted by   $\Psi(Y,X)$. Then $\Psi(Y,X)$ and the shift operators associated with $\Gamma$ generate an algebra  denoted by $\Psi_\Gamma(Y,X)$. This algebra is the central object in this paper.
We define the trajectory symbol mappings on $\Psi_\Gamma(Y,X)$ and characterize the Fredholm property by a suitable notion of ellipticity. For simplicity, we work with matrix operators acting between ranges of projections. Operators acting in sections of vector bundles can be considered similarly, if we realize the sections of vector bundles as ranges of suitable projections. Section~\ref{sec4} contains an application to nonlocal boundary value problems. Namely, we choose a compact submanifold $M\subset Y$  with boundary $\partial M$, consider the algebra $\Psi_\Gamma(Y,X)$, where $X=\Gamma(\partial M)$, and finally study the restriction of $\Psi_\Gamma(Y,X)$ to $M$. We note that elements in this algebra arise in applications, when we make the order reduction for the nonlocal boundary value problems considered in~\cite{BiSo3,Sku1}. As an illustration, we consider the special case of operators with dilations in a unit disc. Section~\ref{sec5} contains the proof of the main results using $C^*$-algebra techniques.  To this end, the spectrum of the norm closure of $\Psi(Y,X)$ is computed. The results in these sections are obtained under the condition that the action of $\Gamma$ on the spectrum of the algebra of symbols is topologically free.
This condition is weakened in Section~\ref{sec6} and an example of dilations of the cylinder is elaborated. There is an appendix at the end of the paper, where necessary results from the theory of $C^*$-algebras are collected.

Our results for  algebras of pseudodifferential operators with shifts are new, while for the case of infinite number of subdomains they give new results also for differential operators with shifts (however, see special case of operators with linear dilations in~\cite{Ross1}). We intend to describe these applications in detail elsewhere. Note also that $C^*$-algebras of pseudodifferential operators are a useful tool in index theory of nonlocal problems, see~\cite{NaSaSt17,SaSchSt2,SaSch1,Per5,BolSa2}.

We are grateful to A.L.~Skubachevskii for numerous fruitful discussions on nonlocal problems.

\section{Boutet de Monvel Type Transmission Problems}\label{sec2}

\subsection*{Main definitions and notation.}
We gather some basic facts about the Boutet de Monvel algebra of transmission problems (see~\cite{Bout2,ReSc1,Schr3,Gru3} and especially~\cite[Section 3.1]{ReSc1}). Let $Y$ be a closed compact smooth manifold and $Z\subset Y$ a closed smooth submanifold of codimension one.  Denote by $Y_Z$ the compact manifold with boundary  obtained by cutting $Y$ along $Z$. The boundary of $Y_Z$  is isomorphic to the unit conormal bundle of $Z\subset Y$. If the normal bundle of $Z$ is trivial, then the boundary of $Y_Z$ is equal to the disjoint union $Z\sqcup Z$.

Consider the Hilbert space
$$
\mathcal{H}=L^2(Y,{\rm{vol}}_Y)\oplus L^2(Z,{\rm{vol}}_Z),
$$
where ${\rm{vol}}_Y$ and ${\rm{vol}}_Z$ are volume forms on $Y$ and $Z$ respectively.
The Boutet de Monvel algebra of operators of order and type zero, denoted in the sequel  by $\Psi(Y,Z)\subset\mathcal{B(H)}$, consists of operators of the form
\begin{equation}
\label{bdm_operator}
\mathcal{D}=
\begin{pmatrix}
D_Y+G_Z & C_Z\\
B_Z &  D_Z
\end{pmatrix}
:\mathcal{H}\longrightarrow\mathcal{H},
\end{equation}
where
\begin{itemize}
\item $D_Y$ and $D_Z$ are pseudodifferential operators ($\psi$DOs below) of order zero on manifolds $Y_Z$ and $Z$ respectively, the operator $D_Y$ satisfies the transmission property (see Remark~\ref{rem1}, below) on the boundary of $Y_Z$;
\item $B_Z$ is a boundary operator on $Z$, $C_Z$ is a coboundary operator on $Z$;
\item $G_Z$ is a (singular) Green operator on $Z$.
\end{itemize}
In a neighborhood of a point $m\in Z$, we choose local coordinates $(x',x_n)$ on $Y$, where $x_n$ is a defining function of $Z$; so that locally $Z=\{x_n=0\}$, while $x'$ stands for local coordinates on $Z$. Denote the dual coordinates in $T^*_mY$ by $(\xi',\xi_n)$.

\begin{remark}\label{rem1}
A $\psi$DO of order $d\in\mathbb{Z}$ with classical symbol $a=a(x',x_n,\xi',\xi_n)$ having an asymptotic expansion
\begin{equation}
\label{asymp_exp}
a\sim a_d+a_{d-1}+\dots
\end{equation}
with homogeneous components $a_l$ satisfies the {\em transmission property}, if
\begin{equation}
\label{trans_prop}
D^k_{x_n}D^\alpha_{\xi'}a_l(x',0,0,\xi_n)=(-1)^{l-|\alpha|}D^k_{x_n}D^\alpha_{\xi'}a_l(x',0,0,-\xi_n),\quad \xi_n\ne 0,
\end{equation}
for all $k\in\mathbb{Z}_+, l\leq d$ and $\alpha=(\alpha_1,\dots,\alpha_{n-1})\in\mathbb{Z}^{n-1}_+$, where
$$
D^\alpha_{\xi'}=\left(-i\frac{\partial }{\partial \xi_1}\right)^{\alpha_1}\cdot \dots \cdot \left(-i\frac{\partial }{\partial \xi_{n-1}}\right)^{\alpha_{n-1}}.
$$
\end{remark}

The {\it conormal bundle} of a submanifold $Z\subset Y$ is the subbundle over $Z$
$$
\mathcal{N}^*Z=\big\{
\xi\in T^*Y|_Z \quad\big|\quad \xi\in T^*_zY \text{ and } \xi(v)=0 \text{ for all } v\in T_z Z\subset T_z Y
\big\}.
$$
By $L^2(\mathcal{N}^*Z)$ we denote a locally trivial bundle over $Z$ with fibers equal to  the Hilbert space $L^2(\mathcal{N}^*_z Z)$, where on $\mathcal{N}^*_z Z$ we consider a measure induced by a Riemannian metric on $Y$. We also denote by $S^*Y_Z$, $S^*Z$ the cosphere bundles of $Y_Z$ and $Z$, given as the quotient of the cotangent bundles with zero section deleted modulo the action of $\mathbb R^*_+$.

Boutet de Monvel operators as in \eqref{bdm_operator} have two symbols:
\begin{itemize}
\item the {\em{interior symbol}} denoted by $\sigma_{\rm{int}}(\mathcal{D})=\sigma(D_Y)\in C^\infty(S^*Y_Z)$;
\item the {\em{boundary symbol}}   denoted by:
$$
\sigma_{Z}(\mathcal{D})=\sigma_{Z}
\left(
\begin{matrix}
D_Y+G_Z & C_Z\\
B_Z & D_Z
\end{matrix}
\right)
\in C^\infty(S^*Z, \End(L^2(\mathcal{N}^*Z)\oplus\mathbb{C})).
$$
Here $\End(L^2(\mathcal{N}^*Z)\oplus\mathbb{C})$ stands for the locally trivial bundle over $S^*Z$ with fiber equal to the algebra of bounded operators $\mathcal{B}(L^2(\mathcal{N}^*_z Z)\oplus\mathbb{C})$ in the indicated Hilbert space.
\end{itemize}

In what follows, we treat functions on $S^*Y_Z$ as homogeneous functions on $T^*Y_Z\setminus 0$, while sections
$a\in C^\infty(S^*Z, \End(L^2(\mathcal{N}^*Z)\oplus\mathbb{C}))$ are treated as  sections on $T^*Z\setminus 0$, which are twisted homogeneous in the following sense:
$$
a(z,\lambda\zeta)= {\kappa}_\lambda^{-1} a(z,\zeta){\kappa}_\lambda,\qquad \forall \lambda\in\mathbb{R}^*_+,(z,\zeta)\in T^*Z\setminus 0,
$$
where we consider unitary operators ${\kappa}_\lambda\in \mathcal{B}(L^2(\mathbb{R})\oplus\mathbb{C})$
\begin{equation}\label{kappalambda}
{\kappa}_\lambda(u(t),v)=(\lambda^{1/2}u(\lambda
t),v).
\end{equation}

\subsection*{Symbols in local coordinates.}
We consider $Y=\mathbb{R}^n, Z=\mathbb{R}^{n-1}=\{x_n=0\}\subset\mathbb{R}^n$. Let $a_\pm \in C^\infty_c(S^*\overline{\mathbb{R}^n_\pm})$ be {\it interior symbols} with the transmission property on $\mathbb{R}^{n-1}$. We extend these functions to $T^*\overline{\mathbb{R}^n_+}\setminus 0$ by homogeneity in $(\xi',\xi_n)$ of degree zero.
Then the boundary symbol
\begin{equation}
\label{bound_symb_trans}
a_Z\in C^\infty_c(S^*\mathbb{R}^{n-1},\mathcal{B}(L^2(\mathbb{R}_{\xi_n})\oplus\mathbb{C}))
\end{equation}
is a smooth matrix operator-function equal to
\begin{equation}
\label{bound_symb_trans_exp}
\begin{split}
a_Z&(x',\xi')\\
&=\begin{pmatrix}
\Pi^+a_+ \Pi^+ + \Pi^-a_- \Pi^- +
\Pi^+_{0,\eta_n}g_+(\xi_n,\eta_n)\Pi^+ +
\Pi^-_{0,\eta_n}g_-(\xi_n,\eta_n)\Pi^-
& c(\xi_n) \vspace{3mm}\\
\Pi^+_0b_+(\xi_n)\Pi^+ + \Pi^-_0b_-(\xi_n)\Pi^- & d
\end{pmatrix},
\end{split}
\end{equation}
where we skip the arguments of $a_\pm$ for brevity.
We use the following notation:
\begin{itemize}
\item $\Pi^\pm=\mathcal{F}\chi_\pm\mathcal{F}^{-1}$ is a projection in $L^2(\mathbb{R}_{\xi_n})$; here $\chi_\pm(x)=1$, if $x\in\mathbb{R}_\pm$, and $0$ otherwise, while $\mathcal{F}:L^2(\mathbb{R}_{x_n})\to L^2(\mathbb{R}_{\xi_n})$ is the Fourier transform;
\item  $a_\pm$ in \eqref{bound_symb_trans_exp} are the multiplication operators obtained by restricting $a_\pm(x',x_n,\xi',\xi_n)$ to $x_n=0$.
\item $H^\pm=\Pi^\pm S(\mathbb{R}_{\xi_n})\subset L^2(\mathbb{R}_{\xi_n})$ and $H=H^+\oplus H^- \stackrel{\mathcal{F}^{-1}}{\simeq} S(\overline{\mathbb{R}_+}) \oplus S(\overline{\mathbb{R}_-})$, where $S(\mathbb{R}), S(\overline{\mathbb{R}_\pm})$ are the Schwartz spaces on $\mathbb{R}$ and $\overline{\mathbb{R}_\pm}$ respectively; these spaces have natural Fr\'echet topologies, and we have an orthogonal decomposition $L^2(\mathbb{R}_{\xi_n})=\overline{H^+}\oplus \overline{H^-}$.

\item Moreover, in \eqref{bound_symb_trans_exp} we use the smooth compactly supported functions
$$
\begin{aligned}
&b_\pm\in C^\infty_c(S^*\mathbb{R}^{n-1},H^\mp), &
&c\in C^\infty_c(S^*\mathbb{R}^{n-1},H),\\
&g_\pm\in C^\infty_c(S^*\mathbb{R}^{n-1},H\otimes H^\mp), &
&d\in C^\infty_c(S^*\mathbb{R}^{n-1});
\end{aligned}
$$

\item  $\Pi^\pm_0:H\to\mathbb{C}$ is the functional given by $\Pi^\pm_0 u=(\mathcal{F}^{-1}u)\big|_{x_n=0\pm}$.
\end{itemize}
Formulas~\eqref{bound_symb_trans} and~\eqref{bound_symb_trans_exp} can be obtained from~\cite[Section 3.1]{ReSc1}. Indeed, in loc. cit. the following unitary isomorphism
\begin{equation*}
\begin{array}{lcl}
\mathcal{H}=L^2(\mathbb{R}^n)\oplus L^2(\mathbb{R}^{n-1})
&\stackrel{\Phi}{\simeq}
&L^2(\mathbb{R}^n_+,\mathbb{C}^2)\oplus L^2(\mathbb{R}^{n-1})\\[2mm]
(u(x',x_n),v(x'))
& \longmapsto
&(u(x',x_n),u(x',-x_n),v(x'))
\end{array}
\end{equation*}
is used to define the transmission problems in $\mathcal{H}$ as matrix Boutet de Monvel problems on the space $L^2(\mathbb{R}^n_+,\mathbb{C}^2)\oplus L^2(\mathbb{R}^{n-1})$. A computation shows that interior and boundary symbols of the operators acting in $L^2(\mathbb{R}^n_+,\mathbb{C}^2)\oplus L^2(\mathbb{R}^{n-1})$ give the interior symbols $a_\pm$ and boundary symbols~\eqref{bound_symb_trans},~\eqref{bound_symb_trans_exp}.

\subsection*{Main properties.}

\begin{prop}
Given $\mathcal{D}_1$ and $\mathcal{D}_2$ in $\Psi(Y,Z)$, we have $\mathcal{D}_1\mathcal{D}_2\in\Psi(Y,Z)$ and the following   composition formulas hold
$$
\sigma_{\rm{int}}(\mathcal{D}_1\mathcal{D}_2)=\sigma_{\rm{int}}(\mathcal{D}_1)\sigma_{\rm{int}}(\mathcal{D}_2);\quad
\sigma_Z(\mathcal{D}_1\mathcal{D}_2)=\sigma_Z(\mathcal{D}_1)\sigma_Z(\mathcal{D}_2).
$$
\end{prop}
We now state the change of variables formula for  operators in $\Psi(Y,Z)$.
Let $\gamma:Y\to Y$ be a diffeomorphism preserving  $Z$, i.e., $\gamma(Z)=Z$. Denote the restriction of $\gamma$ to $Z$ by $\gamma_Z$. Also consider the shift operator
\begin{equation}
\label{shift_op}
T_\gamma: \mathcal{H}\longrightarrow \mathcal{H},\quad (u,v)\longmapsto({\gamma^*}^{-1}u,{\gamma^*_Z}^{-1}v).
\end{equation}
The codifferential of $\gamma$ is denoted by
$$
\partial\gamma:T^*Y\longrightarrow T^*Y
$$
(recall that $\partial\gamma=((d\gamma)^t)^{-1}$, where $d\gamma:TY\to TY$ is the differential, while $(d\gamma)^t:T^*Y\to T^*Y$ is its dual) and it restricts to the conormal bundle $\mathcal{N}^*Z$
$$
\begin{matrix}
\partial \gamma : \mathcal{N}^*Z &\longrightarrow &\mathcal{N}^*Z\\[2mm]
\xi\in \mathcal{N}^*_z Z\subset T^*_zZ& \longmapsto & \partial\gamma(\xi)\in \mathcal{N}^*_{\gamma(z)}Z\subset T^*_{\gamma(z)}Y.
\end{matrix}
$$

\begin{prop}
\label{BdM_symb_thm}
Given $\mathcal{D}\in\Psi(Y,Z)$, we have $T_\gamma\mathcal{D}T^{-1}_\gamma\in \Psi(Y,Z)$   and
\begin{itemize}
\item its interior symbol is equal to
$
\sigma_{\rm{int}}(T_\gamma\mathcal{D}T^{-1}_\gamma)=
(\partial{\gamma)^*}^{-1}(\sigma_{\rm{int}}(\mathcal{D}));
$
\item its boundary symbol is equal to $
\sigma_Z(T_\gamma\mathcal{D}T^{-1}_\gamma)=
\widetilde{\partial\gamma_Z}^{*-1}(\sigma_Z(\mathcal{D})),
$
where the mapping $\widetilde{\partial\gamma_Z}:L^2(\mathcal{N}^*Z)\oplus\mathbb{C}\to L^2(\mathcal{N}^*Z)\oplus\mathbb{C}$ preserves the fibers and is defined as
$$
\big(x',\xi',u\in L^2(\mathcal{N}^*_{x'}Z)\oplus\mathbb{C}\big)
\longmapsto
\big(\partial\gamma_Z(x',\xi'),
{\rm{diag}}(\partial{\gamma^*}^{-1}, 1)u\subset
L^2(\mathcal{N}^*_{\gamma(x')}Z)\oplus\mathbb{C}\big).
$$
\end{itemize}
\end{prop}

\begin{prop}
\label{symb_trans_prop}
If $A\in\Psi(Y)$ is a $\psi$DO with the transmission property on $Z\subset Y$, then $A\in\Psi(Y,Z)$ and its interior and boundary symbols are equal to $\sigma_{\rm{int}}(A)=\sigma(A)$, and $\sigma_Z(A)=\sigma(A)|_{Z}$.
\end{prop}

Let us show that the boundary symbol $\sigma_Z(A)=\sigma(A)|_{Z}$ has the form~\eqref{bound_symb_trans_exp}. We have the equality
\begin{multline}\label{multa1}
a=(\Pi^+ + \Pi^-)a(\Pi^+ + \Pi^-)=\Pi^+a\Pi^+ + \Pi^-a\Pi^- + \Pi^-a\Pi^+ + \Pi^+a\Pi^-
\\
=\Pi^+a\Pi^+ + \Pi^-a\Pi^- +
\Pi^+_{0,\eta_n}\frac{a^-(\xi_n)-a^-(\eta_n)}{i(\xi_n-\eta_n)}\Pi^+ +
\Pi^-_{0,\eta_n}\frac{a^+(\xi_n)-a^+(\eta_n)}{i(\xi_n-\eta_n)}\Pi^-
\end{multline}
of operators acting in $L^2(\mathbb{R})$, where we consider functions $a=\sigma(A)$, $a^\pm=\Pi^\pm a$. Also, the last equality in \eqref{multa1} follows from the identities
$$
\Pi^-a\Pi^+=\Pi^+_{0,\eta_n}\frac{a^-(\xi_n)-a^-(\eta_n)}{i(\xi_n-\eta_n)}\Pi^+,\quad
\Pi^+a\Pi^-=\Pi^-_{0,\eta_n}\frac{a^+(\xi_n)-a^+(\eta_n)}{i(\xi_n-\eta_n)}\Pi^-
$$
proved in~\cite[Section 2.1.2.3, proof of Lemma 5]{ReSc1}, see also~\cite{Gru3}.

\begin{prop}
Given $\mathcal{D}\in\Psi(Y,Z)$, we have the equality
\begin{equation}
\label{estimation_D}
\inf\limits_{K\in \mathcal{K}}\|\mathcal{D}+K\|_{\mathcal{BH}}=
\max\left(
\sup\limits_{(x,\xi)\in S^*Y_Z}|\sigma_{\rm{int}}(\mathcal{D})(x,\xi)|,
\sup\limits_{(x',\xi')\in S^*Z}\|\sigma_Z(\mathcal{D})(x',\xi')\|
\right)
\end{equation}
(norm modulo compact operators), where $\mathcal{K}\subset \mathcal{B(H)}$ is the ideal of compact operators.
\end{prop}

\subsection*{Remark.}
In what follows, we also consider the case, when $Y$ is noncompact.
In this case we consider operators on $\mathcal{H}$ equal to sums of scalar operators and transmission problems with compactly supported Schwartz kernels.

\section{$\Gamma$-Boutet de Monvel transmission problems}\label{sec3}

\subsection*{Manifolds.}
Let $Y$ be a smooth manifold without boundary (noncompact, in general) and $\Gamma\subset{\rm Diff}(Y)$ be a discrete finitely generated group of diffeomorphisms of $Y$. The elements $\gamma\in\Gamma$ are diffeomorphisms denoted by $\gamma:Y\to Y$. We fix a smooth compact submanifold $Z\subset Y$ of codimension one. The partition of $Z$ into connected components is denoted by $Z = \bigcup\limits_{j\in J}Z_j$.

In the sequel, we suppose that the following regularity condition is satisfied.

\begin{cond}
\label{cond1}
Any two submanifolds from the set $\gamma(Z_j), \gamma\in\Gamma, j\in J$ either coincide or are disjoint.
\end{cond}

\begin{example}
Let $Y=\mathbb{S}^1\times\mathbb{R}$ be an infinite cylinder with the coordinates $x_1,x_2$, $Z=\mathbb{S}^1\times\{0\}$, $\Gamma=\mathbb{Z}$ be the group of shifts by $k\tau$, $k\in\mathbb{Z}$, where $\tau > 0$ is a constant. In this case Condition~\ref{cond1} holds.
\end{example}

\begin{example}
Let $Y=\mathbb{R}^d$ and $Z=\mathbb{S}^{d-1}$ be the unit sphere with center at the origin. The group $\Gamma=\mathbb{Z}$ acts by dilations $x\mapsto q^k x$ for a fixed number $q>1$. In this case Condition~\ref{cond1} holds.
More generally one can consider subgroups $\Gamma\subset O(d)\times\mathbb{R}_+$.
\end{example}

\begin{example}
Take $Y=\mathbb{S}^1\times\mathbb{R}$, $Z=\mathbb{S}^1\times\{0\}\cup\mathbb{S}^1\times\{1\}$, $\Gamma=\mathbb{Z}$ acts as $(x_1,x_2)\mapsto (x_1,q^k x_2)$ for a fixed number $q>1$. In this example Condition~\ref{cond1} also holds.
\end{example}

\begin{example}
Take $Y=\mathbb{S}^1\times\mathbb{S}^1$ with coordinates $(x_1,x_2)$, $Z=\mathbb{S}^1\times\{0\}$, $\Gamma=\mathbb{Z}$ be a group of shifts $(x_1, x_2) \mapsto (x_1, x_2+k), k\in\mathbb{Z}, 0\leq x_1,x_2\leq 2\pi$. Then the set $\Gamma(Z)$ is everywhere dense in $Y$. Condition~\ref{cond1} is satisfied in this example.
\end{example}

\subsection*{Algebras of smooth functions.} Define the subset $X=\Gamma (Z)=\bigcup\limits_{\gamma\in\Gamma}\gamma(Z)\subset Y$. This set is a countable union of disjoint submanifolds of codimension one. We fix a word metric $|\gamma|$ on $\Gamma$ and consider the filtration
\begin{equation}
\label{filtration}
X_1\subset X_2\subset\dots \subset X,\text{ where }X_N=\bigcup\limits_{|\gamma|\leq N}{\gamma(Z)}.
\end{equation}
Then $X$ is a smooth manifold with the following topology: $X$ is represented as a disjoint union of different submanifolds $\gamma(Z_j), j\in J, \gamma\in\Gamma$. In other words, a set $U \subset X$ is open if and only if $U \cap X_N$ is open for each $N$. Denote by $C^\infty_c(X)$ the algebra of smooth compactly supported functions on $X$.

On the complement $Y\setminus X_N$, consider the smooth functions, which are smooth up to the submanifold $X_N\subset Y$ from each side of this submanifold.
These functions form an algebra, canonically isomorphic to the algebra $C^\infty(Y_N)$ of smooth functions on the manifold with boundary $Y_N$, obtained from $Y$ by cutting along the submanifold $X_N$ ($Y_N$ is often called the {\it blowup} of $Y$ along $X_N$). We have natural projections $Y_k\to Y_{k'}$ whenever $k>k'$ and the corresponding embeddings
$$
C^\infty_c(Y)\subset C^\infty_c(Y_1)\subset C^\infty_c(Y_2)\subset \dots \subset
C^\infty_c(Y_\infty)\stackrel{\rm{def}}{=}\bigcup\limits_{k}C^\infty_c(Y_k).
$$
The action of $\Gamma$ on $Y$ induces the action of this group on $C^\infty_c(Y_\infty)$ by shift operators (cf.~\eqref{shift_op}).

\subsection*{$L^2$-spaces and unitary shift operators.} We fix a Riemannian metric on $Y$, consider the induced Riemannian metric on $X$, and the corresponding volume forms ${\rm vol}_Y, {\rm vol}_X$ on $Y$ and $X$ respectively. Moreover we define the Hilbert space
$$
\mathcal{H}=L^2(Y,{\rm vol}_Y)\oplus L^2(X,{\rm vol}_X)
$$
using these volume forms.
We have $C^\infty_c(Y_\infty)\subset \mathcal{B}(L^2(Y,{\rm vol}_Y))$.

We introduce the unitary representation $\rho$ of $\Gamma$
\begin{equation}
\label{unit_repres}
\begin{matrix}
\rho:\Gamma&\longrightarrow& \mathcal{B}( \mathcal{H})\\[2mm]
\gamma&\longmapsto&\widetilde{T}_\gamma(u,v)=
(J^{1/2}_{\gamma,Y}T_\gamma u, J^{1/2}_{\gamma,X}T_\gamma v),
\end{matrix}
\end{equation}
where the Jacobians $J_{\gamma,Y}\in C^\infty(Y)$, $J_{\gamma,X}\in C^\infty(X)$ are equal to
$$
J_{\gamma,W}=\frac{{\gamma^*}^{-1}{\rm vol}_Y}{{\rm vol}_Y},\qquad
J_{\gamma,X}=\frac{{\gamma^*}^{-1}{\rm vol}_X}{{\rm vol}_X}.
$$
The unitarity of representation \eqref{unit_repres} can be checked by a direct computation.

\subsection*{Boutet de Monvel operators for the pair $(Y,X)$.}
Given $j\ge 1$, we denote by $\Psi_X(Y,X_j)\subset \Psi(Y,X_j)$ the subalgebra of operators such that their main operators have the transmission property on $X\setminus X_j$. Then the sequence \eqref{filtration} induces a sequence of Boutet de Monvel algebras
$$
\Psi_X(Y,X_1)\subset \Psi_X(Y,X_2)\subset \cdots \subset \Psi(Y,X)\stackrel{def}{=}\bigcup\limits_{j\geq1}\Psi_X(Y,X_j).
$$
Here we used Proposition~\ref{symb_trans_prop}.

Consider the commutative $C^*$-algebra in $L^\infty(S^*Y)$ obtained as the closure of the subalgebra of interior symbols of operators in $\Psi(Y,X)$. By Gelfand's theorem this algebra is the algebra of continuous functions vanishing at infinity on a locally compact Hausdorff space denoted by $S^\vee Y_\infty$. We use ``$\vee$'' in the notation to emphasize that interior symbols are required to have the transmission property on $X$. The one-point compactification of this space is denoted by $S^\vee Y^+_\infty$.

Consider the  union $S^\vee Y^+_\infty\sqcup S^*X$ of disjoint sets. We endow this union with the following topology: the closure $cl(U\sqcup V)$ of a set
$$
U\sqcup V\subset S^{\vee}Y^+_\infty\sqcup S^*X,\quad\text{where}\quad
U\subset S^\vee Y^+_\infty\quad\text{and}\quad
V\subset S^*X,
$$
is equal to the set
\begin{equation}
\label{clos_un}
cl(U\sqcup V) = (\overline{U}\cup\overline{V_Y})\sqcup\overline{V},
\end{equation}
where $\overline{U},\overline{V}$ stand for the closures in the subsets $S^{\vee}Y^+_\infty$ and $S^*X$, respectively,  while
\begin{equation}
\label{V_Y_def}
V_Y=\{
(x',0,(\cos\varphi)\xi',(\sin\varphi)\xi_n)\in S^\vee Y_\infty\quad|\quad (x',\xi')\in{V},\quad \varphi\in[-\pi/2,\pi/2]
\}.
\end{equation}
%
%
We obtain a topological space denoted by $\mathfrak{Y}_X$. Note that this space is not Hausdorff.

\subsection*{$\Gamma$-Boutet de Monvel operators for the pair $(Y,X)$ and their trajectory symbols.}
Let $\Gamma$ act on the algebras $\Psi(Y,X)$ and $C_c^\infty(Y_\infty)\oplus C_c^\infty(X)$ by automorphisms. We denote the corresponding algebraic crossed products by $\Psi(Y,X){\rtimes}_{\rm{alg}}\Gamma$ and $(C_c^\infty(Y_\infty)\oplus C_c^\infty(X))	{\rtimes}_{\rm{alg}}\Gamma$, respectively.
An element $D\in\Psi(Y,X){\rtimes}_{\rm{alg}}\Gamma$ is by definition a compactly supported function $\Gamma\to\Psi(Y,X), \gamma\mapsto D_\gamma.$ We assign to this function the following bounded operator
\begin{equation}
\label{DT_gamma}
D=\sum\limits_\gamma{D_\gamma\widetilde{T}_\gamma}:\mathcal{H}\longrightarrow\mathcal{H}.
\end{equation}
The operator~\eqref{DT_gamma} is called a {\it  $\Gamma$-Boutet de Monvel operator for the pair $(Y,X)$}.

We define the {\it interior trajectory symbol}
\begin{align}
\label{Gamma_BdM_int_traj_symb}
&\sigma_{\rm{int}}(D)(x,\xi)\colon
l^2(\Gamma)\longrightarrow l^2(\Gamma),\quad(x,\xi) \in
S^*Y^+_\infty,
\intertext{and the \textit{boundary trajectory symbol}}
\label{Gamma_BdM_boun_traj_symb}
&\sigma_X(D)(x',\xi')\colon l^2(\Gamma,
L^2(\mathcal{N}^*_{x'}X) \oplus
\mathbb{C})\longrightarrow l^2(\Gamma,
L^2(\mathcal{N}^*_{x'}X) \oplus \mathbb{C}),\quad
(x',\xi')\in S^*X,
\end{align}
of the operator~\eqref{DT_gamma} as
\begin{align}\notag
[\sigma_{\rm{int}}(D)(x,\xi)](w)(\gamma_1)
&=\sum\limits_\gamma\sigma_{\rm{int}}
(D_\gamma)((\partial\gamma_1)^{-1}(x,\xi))
w(\gamma_1\gamma),\quad w\in l^2(\Gamma),
\\ \label{eq-bnd-trj1}
[\sigma_X(D)(x',\xi')](w)(\gamma_1)
&=\sum\limits_\gamma
[{\widetilde{\widetilde{\partial\gamma_1}}}^{*^{-1}}
(\sigma_X(D_\gamma))](x',\xi') w(\gamma_1\gamma),
\\ \notag &\qquad\qquad w\in l^2(\Gamma,
L^2(\mathcal{N}^*_{x'}X)\oplus\mathbb{C}),
\end{align}
where $\widetilde{\widetilde{\partial\gamma}}:L^2(\mathcal{N}^*X) \oplus \mathbb{C}\to L^2(\mathcal{N}^*X)\oplus \mathbb{C}$ is equal to (cf.~Proposition~\ref{BdM_symb_thm})
$$
\widetilde{\widetilde{\partial\gamma}}(x',\xi',u) =(\partial\gamma(x',\xi'),{\rm diag}(J_N^{1/2}(\partial\gamma)^{*^{-1}},1)u), \qquad J_N=\frac{(\partial\gamma)^{*^{-1}} \operatorname{vol}_{\mathcal{N}^*X}}{\operatorname{vol}_{\mathcal{N}^*X}}.
$$
It follows from the change of variables formula for Boutet de Monvel operators (see Proposition~\ref{BdM_symb_thm}) that the trajectory symbols~\eqref{Gamma_BdM_int_traj_symb} and~\eqref{Gamma_BdM_boun_traj_symb} define homomorphisms of algebras
\begin{align}
\label{int_traj_symb_alg_hom}
\sigma_{\rm{int}}(x,\xi):\Psi(Y,X){\rtimes}_{\rm{alg}}\Gamma &\longrightarrow  \mathcal{B}(l^2(\Gamma))
\\[2mm]
\label{boun_traj_symb_alg_hom}
\sigma_X(x',\xi'):\Psi(Y,X){\rtimes}_{\rm{alg}}\Gamma& \longrightarrow \mathcal{B}(l^2(\Gamma, L^2(\mathcal{N}^*_{x'}X)\oplus\mathbb{C}))
\end{align}
for all $(x,\xi)\in S^*Y^+_\infty$, $(x',\xi')\in S^*X$.

\subsection*{Ellipticity and Fredholm property.}

Consider triples of the form $\mathcal{D}=(D, P_1, P_2)$, where
$$
D\in{\rm{Mat}}_N(\Psi(Y,X){\rtimes}_{\rm{alg}}\Gamma),\qquad
P_{1,2}\in{\rm{Mat}}_N((C^\infty_c(Y_\infty)\oplus C^\infty_c(X))^+{\rtimes}_{\rm{alg}}\Gamma).
$$
Here $P_{1,2}$ are projections $(P_j^2=P_j)$, and we consider the unitization of $C^\infty_c(Y_\infty)\oplus C^\infty_c(X)$. We consider matrix operators over the corresponding algebras. Moreover, we suppose that the following equality holds:
$D=P_2DP_1$. This implies that $D$ restricts to the operator
\begin{equation}
\label{D_restr}
D\colon P_1(\mathcal{H}\otimes\mathbb{C}^N)\longrightarrow P_2(\mathcal{H}\otimes\mathbb{C}^N),
\end{equation}
acting between the ranges of $P_1$ and $P_2$. We call the operator \eqref{D_restr} a $\Gamma$-{\em{pseu\-do\-dif\-fe\-ren\-ti\-al operator}} or shortly a $\Gamma$-{\em{operator}}. Let us give the ellipticity conditions for $\Gamma$-operators, which guarantee their Fredholm property.

\begin{definition}
A triple $(D, P_1, P_2)$ is {\em trajectory elliptic}, if the following two conditions are satisfied:
\begin{enumerate}
\item[1)] the interior trajectory symbol (see~\eqref{int_traj_symb_alg_hom})
induces an invertible operator
\begin{equation}
\label{traj_ell_in}
\sigma_{\rm{int}}(D)(m):
P_1(m) l^2(\Gamma,\mathbb{C}^N) \longrightarrow
P_2(m) l^2(\Gamma,\mathbb{C}^N)
\end{equation}
for all $m\in S^*Y^+_\infty$; here $P_j(m)$ is the interior trajectory symbol of $P_j$;
\item[2)] the boundary trajectory symbol (see~\eqref{boun_traj_symb_alg_hom})
induces an invertible operator
\begin{multline}
\label{traj_ell_b}
\sigma_X(D)(m'):
P_1(m') l^2(\Gamma,(L^2( \mathcal{N}^*_{m'}X)\oplus\mathbb{C})\otimes\mathbb{C}^N) \longrightarrow \\
P_2(m') l^2(\Gamma,(L^2( \mathcal{N}^*_{m'}X)\oplus\mathbb{C})\otimes\mathbb{C}^N)
\end{multline}
for all $m'\in S^*X$; here $P_j(m')$ is the boundary trajectory symbol of $P_j$.
\end{enumerate}
\end{definition}

Recall that an action of a finitely generated group $\Gamma$ by homeomorphisms  on a topological space $W$ is topologically free if for any finite collection of group elements $\gamma_1,\dots,\gamma_N\in\Gamma\setminus\{e\}$ the union of their fixed point sets  does not contain a nonempty open set in $W$.

\begin{theorem}
\label{tr_kon}
Suppose that $\Gamma$ is amenable and its action on the space $\mathfrak{Y}_X$  is topologically free (the topology on this space is defined in~\eqref{clos_un}). Then the operator~\eqref{D_restr} has the Fredholm property if and only if the triple $(D, P_1, P_2)$ is trajectory elliptic.
\end{theorem}
The proof of this theorem is given in Section~\ref{sec5} using the theory of $C^*$-algebras.

\section{Application to Nonlocal Boundary Value Problems}\label{sec4}
Let $Y$ be as above and $M\subset Y$  a compact submanifold of codimension zero with boundary denoted by $Z$.
We now show that nonlocal boundary value problems in $M$ can be defined and studied using the results obtained above.

We suppose that the pair $(Y,Z)$ satisfies Condition~\ref{cond1}. Consider the projections
\begin{equation}
\label{proj_char}
P_{1,2}=(\chi_M,\chi_{W_{1,2}})\in C^\infty_c(Y_\infty)\oplus C^\infty_c(X),
\end{equation}
where $\chi_M$ is the characteristic function of $M$ and $\chi_{W_{1,2}}$ are the characteristic functions of some submanifolds $W_{1,2}\subset X\cap M$ equal to a finite union of submanifolds $\gamma(Z_j)$, where, as above, $\gamma\in\Gamma$ and $Z_j, j\in J$ denote the connected components of $Z$. Then, given an operator $D\in\Psi(Y,X){\rtimes}_{\rm{alg}}\Gamma$, we consider the $\Gamma$-operator
$$
P_2DP_1:L^2(M)\oplus L^2(W_1)\longrightarrow L^2(M)\oplus L^2(W_2).
$$
Theorem~\ref{tr_kon} gives a criterion for the Fredholm property of this operator in terms of the trajectory symbols of the operator $D$. Let us state these conditions in the following interesting example.

\begin{example}
Consider $Y=\mathbb{R}^d$ as ambient manifold  and let $M=\mathbb{B}^d$ be the unit ball.
Then $\Gamma=\mathbb{Z}$ acts by dilations $x\mapsto q^k x$ for a fixed number $0<q<1$. In this case, Condition~\ref{cond1} holds.
We take the projection
$$
\Pi=(\chi_{\mathbb{B}^d},\chi_{\mathbb{S}^{d-1}})\in C^\infty_c(Y_\infty)\oplus C^\infty_c(X)
$$
on the subspace
$$
L^2(\mathbb{B}^d)\oplus L^2(\mathbb{S}^{d-1})\subset
L^2(\mathbb{R}^d)\oplus L^2\Bigl(\bigcup\limits_{k\in\mathbb{Z}}q^k\mathbb{S}^{d-1}\Bigr)
$$
and consider the following $\Gamma$-operator:
\begin{equation}
\label{BdM_example}
\mathcal{D}=\Pi\sum\limits_{k\in\mathbb{Z}}D_k\widetilde{T}^k:
L^2(\mathbb{B}^d)\oplus L^2(\mathbb{S}^{d-1})\longrightarrow
L^2(\mathbb{B}^d)\oplus L^2(\mathbb{S}^{d-1}),
\end{equation}
where $D_k\in\Psi(Y,X)$, while
$$
\widetilde{T}:L^2(\mathbb{R}^d)\oplus L^2\Bigl(\bigcup\limits_{k\in\mathbb{Z}}q^k\mathbb{S}^{d-1}\Bigr)\longrightarrow
L^2(\mathbb{R}^d)\oplus L^2\Bigl(\bigcup\limits_{k\in\mathbb{Z}}q^k\mathbb{S}^{d-1}\Bigr)
$$
is the unitary operator
$$
\widetilde{T}(u(x),\{v_k(x)\}_{k\in\mathbb{Z}})=\bigl(q^{-d/2}u(q^{-1}x),\{ q^{-(d-1)/2}v_{k-1}(q^{-1}x)\}\bigr).
$$
Then we obtain the following trajectory symbols of the operator~\eqref{BdM_example}:
\begin{enumerate}
\item[1)] The {\it interior trajectory symbol} at the points $(x,\xi)\in S^*\mathbb{B}^d, x\ne0$ is given by
\begin{equation}
\label{int_tr_symb_2}
\begin{aligned}
\sigma_{\rm{int}}(\mathcal{D})(x,\xi)&\colon
l^2(\mathbb{Z}_{\leq N})\longrightarrow
l^2(\mathbb{Z}_{\leq N}), \quad N= \frac{\ln |x|}{\ln
{q}}\geq 0
\\
(\sigma_{\rm{int}}(\mathcal{D})(x,\xi)w)(n)&=
\sum\limits_{k\leq
N-n}\sigma_{\rm{int}}(D_k)(q^{-n}x,\xi)w(n+k).
\end{aligned}
\end{equation}
\item[2)] The {\it interior trajectory symbol} at the points $(0,\xi)\in S^*\mathbb{B}^d$ is given by
\begin{equation}
\label{int_tr_symb_1}
\sigma_{\rm{int}}(\mathcal{D})(0,\xi)=\sum\limits_k\sigma_{\rm{int}}(D_k)(0,\xi)\mathcal{T}^k:l^2(\mathbb{Z})\longrightarrow l^2(\mathbb{Z});
\end{equation}
here $\mathcal{T}w(n)=w(n+1)$ is the shift operator on sequences; note that~\eqref{int_tr_symb_2} tends to~\eqref{int_tr_symb_1} as $x\to 0$.
\item[3)] The  {\it boundary trajectory symbol} at the points $(x',\xi')\in S^*\mathbb{S}^{d-1}$ is
\begin{equation*}
\sigma_{\mathbb{S}^{d-1}}(\mathcal{D})(x',\xi')
\in\mathcal{B}(\mathcal{H}_0),\quad\text{where }
\mathcal{H}_0=p\Bigl(\bigoplus\limits_{n\in\mathbb{Z}} (L^2(\mathbb{R})\oplus\mathbb{C})\Bigr).
\end{equation*}
Here the projection $p$ is the symbol of $\Pi$ and is equal to
$$
(u,v)=(\ldots,u_{-1},v_{-1},u_0,v_0,u_1,v_1,\ldots)\stackrel{p}\longmapsto
(\ldots,u_{-1},0,\Pi^-u_0,v_0,0,0,\ldots).
$$
Then \eqref{eq-bnd-trj1} and \eqref{traj_ell_b} give the following expression for the desired boundary symbol:
\begin{equation*}
(\sigma_{\mathbb{S}^{d-1}}(\mathcal{D})(x',\xi')(u,v))(n)=
p
\Bigl(\sum\limits_{k\leq-n}{\kappa}_{q^n}\sigma_{q^{-n}\mathbb{S}^{d-1}}(D_k)(q^{-n}x',\xi'){\kappa}^{-1}_{q^n}
\begin{pmatrix}
u_{n+k}\\
v_{n+k}
\end{pmatrix}
\Bigr).
\end{equation*}
\end{enumerate}
Here we have used the unitary operators ${\kappa}_\lambda$ defined in \eqref{kappalambda}.
\end{example}

\section{Proofs of the Main Results}\label{sec_C*_theory}\label{sec5}
The aim of this section is to prove Theorem~\ref{tr_kon} using the theory of $C^*$-algebras. To this end, we pass to norm closures of the algebras introduced in Section 3.

\subsection*{Continuous functions on $Y_\infty$.}
The norm closure of $C^\infty_c(Y_\infty)\subset L^\infty(Y)$ is the algebra $C_0(Y_\infty)$ of continuous functions on the  locally compact Hausdorff space $Y_\infty$, which tend to zero at infinity (the Hausdorff property and local compactness follow from Gelfand's theorem). The homomorphism of algebras $C_0(Y)\to C_0(Y_\infty)$ is induced by the projection $\pi:Y_\infty\to Y$. Similarly, the closure of $C^\infty_c(X)\subset L^\infty(X)$ is equal to $C_0(X)$.

We define the action of $\Gamma$ on $C_0(Y_\infty)$ and $C_0(X)$ by shift operators
$$
(T_\gamma u)(x) =u(\gamma^{-1}x), \quad\text{for all}\quad \gamma\in\Gamma.
$$

\subsection*{Example.} Let $Y=\mathbb{S}^1\times\mathbb{S}^1$ with coordinates $x_1,x_2$; $\Gamma=\mathbb{Z}\subset{\rm{Diff}}(Y)$ be the group of translations $(x_1,x_2)\mapsto (x_1,x_2+k)$, $Z=\mathbb{S}^1=\{(x_1,0)\}\subset Y$. Let us describe the commutative $C^*$-algebra $C(Y_\infty)$.
\begin{prop}
We define the commutative $C^*$-subalgebra $\mathcal{A}\subset L^\infty(\mathbb{S}^1\times\mathbb{S}^1)$ by
$$
\mathcal{A}=
\left\{
f:\mathbb{S}^1\times\mathbb{S}^1\longrightarrow\mathbb{C}
\left|
\begin{array}{l}
1)\;f \text{ is continuous  at all points }(x_1,x_2), x_2\ne k,\; k\in\mathbb{Z};\\[2mm]
2)\;f\text{ is left and right continuous } \\
 \quad\text{ at the points }(x_1,x_2),\;x_2=k,\; k\in\mathbb{Z};\\
3)\; \lim\limits_{k\to\infty}\Bigl(\max\limits_{x_1\in\mathbb{S}^1}
\left|f(x_1,k+0)-f(x_1,k-0)\right|\Bigr)=0
\end{array}
\right.
\right\}.
$$
Then we have $C(Y_\infty)=\mathcal{A}$.
\end{prop}
\begin{proof}
1. Note that $\mathcal{A}\subset L^\infty(\mathbb{S}^1\times\mathbb{S}^1)$ is a closed subspace. This readily follows, since the properties in the definition of $\mathcal{A}$ are closed with respect to the $\sup$ norm.

2. We also have $C(Y_N)\subset\mathcal{A}$ for all $N$.

3. To prove that $C(Y_\infty)=\mathcal{A}$, it suffices by 1. and 2. to show that $\bigcup\limits_{N}C(Y_N)$ is dense in $\mathcal{A}$. Indeed, given $f\in\mathcal{A}$ and $\varepsilon>0$, there exists $N$ such that:
\begin{equation}
\label{max_diff}
\max\limits_{x_1\in\mathbb{S}^1}\left|f(x_1,k+0)-f(x_1,k-0)\right|<\varepsilon
\text{ whenever }|k|\geq N.
\end{equation}
Denote by $\mathbb{S}^1_N$ the space obtained by cutting $\mathbb{S}^1$ at the points $\{j\;\big|\; |j|\leq N\}$. Topologically, this space is a disjoint union of $2N+1$ segments. It follows from the continuity of $f(x_1,x_2)$ at $x_2\notin\mathbb{Z}$, one-sided continuity at $x_2\in\mathbb{Z}$, and~\eqref{max_diff} that, given any $x_2\in\mathbb{S}^1_N$, there exists $\delta=\delta(x_2)$ such that
$$
|f(x_1,x_2)-f(x_1,x'_2)|<2\varepsilon,\text{ for all }(x_1,x_2) \text{ such that } |x_2-x'_2|<\delta.
$$
Intervals of the form $\{x_2\;\big|\; |x_2-x'_2|<\delta(x'_2)\}$, where $x_2'$ runs over $\mathbb{S}^1_N$, form an open cover of $\mathbb{S}^1_N$. By compactness, we can find a finite subcover and define a finite set $\mathcal{F}\subset \mathbb{S}^1_N$ of centers of intervals in our subcover. Given $x_2\in\mathbb{S}^1_N$, we denote by $x^\pm_2\in\mathcal{F}$ the closest element to $x_2$ from the right/left.

We define a function $f_\varepsilon$ that is piecewise linear in $x_2$ by
$$
f_\varepsilon(x_1,x_2)=f(x_1,x^-_2)(1-\alpha(x_2))+f(x_1,x^+_2)\alpha(x_2),\text{ where } \alpha(x_2)=\frac{x_2-x^-_2}{x^+_2-x^-_2}.
$$
By construction $f_\varepsilon\in C(Y_N)$ and
\begin{align*}
\|f-f_\varepsilon\|
&=\sup\limits_{x_1,x_2}|f(x_1,x_2)-f_\varepsilon(x_1,x_2)|
\\
&=
\sup\limits_{x_1,x_2}|f(x_1,x_2)(1-\alpha(x_2)+\alpha(x_2))-
f(x_1,x^-_2)(1-\alpha(x_2))-f(x_1,x^+_2)\alpha(x_2)|
\\
&\leq
\sup\limits_{x_1,x_2}|f(x_1,x_2)-f(x_1,x^-_2)| +
\sup\limits_{x_1,x_2}|f(x_1,x_2)-f(x_1,x^+_2)|\leq
4\varepsilon.
\end{align*}
This shows that $f\in \overline{\bigcup\limits_N C(Y_N)}$.
\end{proof}

\begin{remark}
One can obtain a similar description of $C_0(Y_\infty)$ in the general case.
\end{remark}
\subsection*{The $C^*$-algebra of Boutet de Monvel problems.}
In the notation of Section~2 we denote by $\overline{\Psi(Y,Z)}\subset\mathcal{B(H)}$ the norm closure of Boutet de Monvel algebra.
It follows from \eqref{estimation_D} that the mapping which takes a Boutet de Monvel operator to its interior and boundary symbols, extends by continuity to an injective homomorphism of $C^*$-algebras
$$
(\sigma_{\rm{int}}, \sigma_Z):
\overline{\Psi(Y,Z)}/\mathcal{K}\longrightarrow
C(S^{\vee}Y^+_Z)\oplus
C(S^*Z,\End(L^2(\mathcal{N}^*Z)\oplus\mathbb{C})).
$$
Here $S^{\vee}Y_Z$ stands for the topological space obtained from the cosphere bundle $S^*Y_Z$ by identification of the pairs of points $(x',0,0,\pm 1)$, where $x'$ runs over $\partial Y_Z$, and $S^{\vee}Y^+_Z$ is its one-point compactification.

We have a short exact sequence of $C^*$-algebras:
\begin{equation}
\label{short_seq_1}
0 \longrightarrow C(S^*Z, \mathcal{K}) \stackrel{\alpha}{\longrightarrow}
\overline{\Psi(Y,Z)}/\mathcal{K} \stackrel{\sigma_{\rm{int}}}{\longrightarrow} C(S^{\vee}Y^+_Z)\longrightarrow 0,
\end{equation}
where $\alpha$ takes  a compact-valued boundary symbol  in $C(S^*Z,\End(L^2(\mathcal{N}^*Z)\oplus\mathbb{C}))$ to a Boutet de Monvel operator with this boundary symbol and interior symbol equal to zero.

\subsection*{The spectrum of the Boutet de Monvel $C^*$-algebra.} The Calkin algebra $\overline{\Psi(Y,Z)}/\mathcal{K}$ of the Boutet de Monvel algebra is denoted by $\Sigma$. An operator $[\mathcal{D}]\in \overline{\Psi(Y,Z)}/\mathcal{K}$ has symbols at the corresponding points of the cosphere bundles, and we define the representations
\begin{equation}
\label{repr_symb}
\begin{matrix}
\sigma_{\rm{int}}(m):\Sigma\longrightarrow\mathbb{C},&
\mathcal{D}\longmapsto\sigma_{\rm{int}}(\mathcal{D})(m), & m\in S^{\vee}Y^+_Z,
\\[3mm]
\sigma_Z(m'):\Sigma\longrightarrow\End(L^2(\mathcal{N}_{m'}^*Z)\oplus\mathbb{C}), &
\mathcal{D}\longmapsto\sigma_Z(\mathcal{D})(m'), & m'\in S^*Z.
\end{matrix}
\end{equation}
It is easy to show that these representations are irreducible (this follows from the fact that the images of these representations contain the ideal of compact operators). It turns out that~\eqref{repr_symb} exhaust all irreducible representations of the considered Calkin algebra. More precisely, the following result holds. We recall all necessary facts from the theory of $C^*$-algebras from~\cite{Dix1} in the appendix.
\begin{prop}
\label{prop_GCR}
The algebra $\Sigma=\overline{\Psi(Y,Z)}/\mathcal{K}$ is a GCR-$C^*$-algebra and the following homeomorphism of topological spaces is valid
\begin{equation}
\label{symb_homeom_1}
\begin{split}
 \mathfrak{Y}_Z&\simeq \Prim\Sigma,\\
(m, m')&\longmapsto (\sigma_{\rm{int}}(m),
\sigma_Z(m')),
\end{split}
\end{equation}
where $\Prim\Sigma$ stands for the space of primitive ideals in $\Sigma$ endowed with Jacobson topology (e.g.~\cite[Definition 3.1.1]{Dix1}), and the topology on $\mathfrak{Y}_Z$ is defined in~\eqref{clos_un} and~\eqref{V_Y_def}.
\end{prop}
\begin{proof}
The proof is similar to that in~\cite[Theorem~60.2]{AnLe2}.

1. Consider the composition series
$$
I_0=0\subset I_1=C(S^*Z, \mathcal{K})\subset I_2=\Sigma.
$$
From the exact sequence~\eqref{short_seq_1} we obtain the isomorphism $I_2/I_1\simeq C(S^\vee Y^+_Z)$ of $C^*$-al\-geb\-ras. Thus, $I_2$ is a GCR-$C^*$-algebra by~\cite[Proposition 4.3.4]{Dix1} and we have the bijections (see~\cite[Proposition 2.11.5]{Dix1})
$$
\Prim I_2=\Prim(I_2/I_1)\sqcup \Prim I_1=\mathfrak{Y}_Z,
$$
i.e., the primitive ideals in $\Sigma$ are exhausted by the kernels of irreducible representations~\eqref{repr_symb}. This proves that~\eqref{symb_homeom_1} is a bijection.

2. We now describe the topology on $\Prim I_2$. By~\cite[Proposition 3.2.1]{Dix1} the set $\Prim_{I_1} (I_2)\simeq S^\vee Y^+_Z$ is closed in $\Prim I_2$, while $\Prim^{I_1} (I_2)\simeq S^*Z$ is open in $\Prim I_2$.

Let $V\subset\Prim I_2$ be a set. We show that its closure can be described by~\eqref{clos_un}, \eqref{V_Y_def}. It is sufficient to consider two cases: when $V\subset S^\vee Y^+_Z$ and when $V\subset S^*Z$.

Firstly, let $V\subset S^\vee Y^+_Z$. Then, since $S^\vee Y^+_Z$ is closed in $\Prim I_2$, the closure of $V$ coincides with the closure $\overline{V}\subset S^\vee Y_Z$ in the sense of the topology on $S^\vee Y_Z$, i.e., in this case~\eqref{V_Y_def} is true.

Now let $V\subset S^*Z$. Then by the definition of the Jacobson topology (see~\cite[Theorems 3.4.4 and 3.4.10]{Dix1}) a point $z$ is in $cl(V)$ if and only if one has:
\begin{equation}
\label{ker_inclusion_1}
\bigcap\limits_{x\in V}\ker \pi_x\subset \ker \pi_z,
\end{equation}
where $\pi_x$ denotes the irreducible representation for $x\in\Prim\Sigma$.
More precisely, the set on the left in~\eqref{ker_inclusion_1} is equal to
\begin{equation}
\label{ker_inclusion_left_1}
\bigcap\limits_{x\in V}\ker \pi_x=
\{
[\mathcal{D}]\in I_2\mid \sigma_Z(\mathcal{D})(x)=0\, \text{ for all } x\in V
\}.
\end{equation}
If $z\in cl(V)\cap S^*Z$, then it follows from~\eqref{ker_inclusion_left_1} and the continuity of the boundary symbol $\sigma_Z(\mathcal{D})$ on $S^*Z$ that in this case the inclusion~\eqref{ker_inclusion_1} holds if and only if $z$ is in the closure of $V\subset S^*Z$ in the usual topology. If $z\in cl(V)\cap S^\vee Y^+_Z$, then relation~\eqref{ker_inclusion_left_1} and the compatibility condition of interior and boundary symbols imply that condition~\eqref{ker_inclusion_1} holds for all points $z\in\overline{V_Y}\subset S^\vee Y^+_Z$ (recall that given a set $V\subset S^*Z$ the set $V_Y$ was defined in~\eqref{V_Y_def}). Also, if $z\notin\overline{V_Y}$, then condition~\eqref{ker_inclusion_1} does not hold, since one can construct an operator $\mathcal{D}\in\Psi(Y,Z)$ with the following properties:
\begin{itemize}
\item $\sigma_{\rm{int}}(\mathcal{D})(z)\ne 0$, i.e., $[\mathcal{D}]\notin\ker\pi_z$;
\item $\sigma_Z(\mathcal{D})(m)=0\text{ for all points } m\in V$.
\end{itemize}
The interior symbol of this operator is constructed in the following way: Let $\chi\in C_0(S^\vee Y_Z)$ be a function with the properties $\chi(z)=1$ and $\chi|_{\overline{V_Y}}\equiv 0$. Then we set $\sigma_{\rm{int}}(\mathcal{D})=\chi$.

Thus, we find that the closure of $V\subset S^*Z$ in $\Prim I_2$ is equal to $\overline{V_Y}\sqcup\overline{V}\subset  \mathfrak{Y}_Z$, i.e., in this case~\eqref{clos_un} also holds.

This completes the proof of Proposition~\ref{prop_GCR}.
\end{proof}

\subsection*{The spectrum of the Boutet de Monvel $C^*$-algebra for pairs $(Y,X)$.}

The estimate~\eqref{estimation_D} of the norm modulo compact operators  implies that the symbol mapping
on $\Psi(Y,X)$ extends by continuity to a $C^*$-homomorphism
\begin{equation*}
(\sigma_{\rm{int}},\sigma_X):\overline{\Psi(Y,X)}/\mathcal{K}\longrightarrow
C(S^\vee Y^+_\infty)\oplus C_0(S^*X,\End(L^2(\mathcal{N}^*X)\oplus\mathbb{C}))^+
\end{equation*}
and we define the representations of this algebra
\begin{equation}
\label{repr_symb_alg}
\begin{aligned}
&\begin{aligned}
\sigma_{\rm int}(m): \overline{\Psi(Y,X)}/\mathcal{K} &
\longrightarrow \mathbb{C},\\
\mathcal{D} & \longmapsto \sigma_{\rm
int}(\mathcal{D})(m),
\end{aligned}
&m&\in S^\vee Y^+_\infty;
\\
&\begin{aligned}
\sigma_X(m'): \overline{\Psi(Y,X)}/\mathcal{K} &
\longrightarrow
\mathcal{B}(L^2(\mathcal{N}^*_{m'}X)\oplus\mathbb{C}),\\
\mathcal{D} & \longmapsto \sigma_X(\mathcal{D})(m'),
\end{aligned}
&m'&\in S^*X.
\end{aligned}
\end{equation}
These representations are irreducible.

\begin{lemma}
We have a short exact sequence of $C^*$-algebras
\begin{equation}
\label{short_seq_2}
0 \longrightarrow C_0(S^*X, \mathcal{K})
\xrightarrow{\alpha_\infty}
\overline{\Psi(Y,X)}/\mathcal{K}
\xrightarrow{\sigma_{\rm{int}}}
C(S^{\vee}Y^+_\infty)\longrightarrow 0,
\end{equation}
where $C_0(S^*X, \mathcal{K})\subset C_0(S^*X,\End(L^2(\mathcal{N}^*X)\oplus\mathbb{C}))$ stands for the subspace of functions valued in compact operators and the mapping $\alpha_\infty$ is defined on the dense subset $C_c(S^*X,\mathcal{K})=\bigcup\limits_N C(S^*X_N,\mathcal{K})$ as in~\eqref{short_seq_1}.
\end{lemma}
\begin{proof}
1. The mapping $\alpha_\infty$ is well defined since $\alpha$ in~\eqref{short_seq_1} is norm-preserving. This also implies that $\alpha_\infty$ is injective.

2. Let us prove that $\sigma_{\rm{int}}$ is surjective.
Given $a\in C(S^\vee Y^+_\infty)$ and $a_N\longrightarrow a$ in $L^\infty(S^* Y)$, where $a_N\in C(S^\vee Y^+_N)$, we choose $A_N\in\Psi_X(Y,X_N)$ with
$\sigma_{\rm{int}}(A_N) = a_N$ and boundary, coboundary and Green components equal to zero. Then the norm of $A_N-A_{N'}$ modulo compact operators is equal to $\|a_N-a_{N'}\|$. This implies that $A_N$ converges in $\overline{\Psi(Y,X)}/\mathcal{K}$ to some $B$ as $N\to\infty$. Hence, $\sigma_{\rm{int}}(B)=\lim\limits_{N\to\infty}a_N=a$ and we obtain the desired surjectivity of $\sigma_{\rm{int}}$ in~\eqref{short_seq_2}.

3. Let us prove the exactness of~\eqref{short_seq_2} in the middle term. The inclusion ${\rm{Im}}\,\alpha_\infty\subset{\rm{ker}}\,\sigma_{\rm{int}}$ follows by continuity from the exactness of sequences of the form~\eqref{short_seq_1} with $Z=X_N,$ $N\geq1$. Let us show that ${\rm{ker}}\,\sigma_{\rm{int}}\subset{\rm{Im}}\,\alpha_\infty$. Indeed, given $A=\lim\limits_{N\to\infty}A_N$, $A_N\in\Psi_X(Y,X_N)$, the fact that  $\sigma_{\rm{int}}(A)=0$ implies that $\sigma_{\rm{int}}(A_N)\to 0$. We choose $B_N\in\Psi_X(Y,X_N)$  with $\sigma_{\rm{int}}(B_N)=\sigma_{\rm{int}}(A_N)$ and $\|B_N\|\leq 2\|\sigma_{\rm{int}}(A_N)\|$. Then we have
$$
\lim\limits_{N\to\infty}(A_N-B_N)=\lim\limits_{N\to\infty}A_N-\lim\limits_{N\to\infty}B_N=A,
$$
while $\sigma_{\rm{int}}(A_N-B_N)=0$. By the exactness of~\eqref{short_seq_1} we have $A_N-B_N=\alpha(c_N)$, where $c_N\in C(S^*X_N,\mathcal{K})$ and
$$
\|c_N-c_{N'}\|=\|\alpha(c_N-c_{N'})\|\leq\|A_N-A_{N'}\| + \|B_N-B_{N'}\|.
$$
This estimate and the convergence of $A_N$ and $B_N$ as $N\to\infty$ imply that $c_N\to c$ as $N\to\infty$, where $c\in C_0(S^*X,\mathcal{K}),$ and we have by continuity $\alpha_\infty(c)=A$ in $\overline{\Psi(Y,X)}/\mathcal{K}$.

The proof of the lemma is now complete.
\end{proof}

\begin{prop}
The algebra $\Sigma=\overline{\Psi(Y,X)}/\mathcal{K}$ is a GCR-$C^*$-algebra and we have a homeomorphism of spaces:
\begin{equation}
\label{symb_homeom_2}
\begin{aligned}
\mathfrak{Y}_X &\simeq \Prim\Sigma\\
(m, m')&\longmapsto (\sigma_{\rm{int}}(m),
\sigma_X(m')).
\end{aligned}
\end{equation}
\end{prop}
\begin{proof}
The proof is similar to that of Proposition~\ref{prop_GCR}.

1. First, \eqref{short_seq_2} gives the composition series
$$
I_0=0\subset I_1=C_0(S^*X,\mathcal{K})\subset I_2=\overline{\Psi(Y,X)}/\mathcal{K}
$$
and $I_2/I_1\simeq C(S^\vee Y^+_\infty)$. Hence, $I_2$ is a GCR-$C^*$-algebra and
$$
{\rm{Prim}}I_2={\rm{Prim}}(I_2/I_1)\sqcup{\rm{Prim}}I_1=\mathfrak{Y}_X,
$$
i.e., the primitive ideals in $\overline{\Psi(Y,X)}/\mathcal{K}$ are  exhausted by the kernels of the irreducible representations~\eqref{repr_symb_alg}, and~\eqref{symb_homeom_2} is a bijection.

2. The proof that the Jacobson topology on ${\rm{Prim}}(\overline{\Psi(Y,X)}/\mathcal{K})$ coincides with that described in~\eqref{clos_un},\eqref{V_Y_def} repeats the argument in the proof of Proposition~\ref{prop_GCR} if we replace $Z$ by $X$ and use Urysohn's lemma to construct the function $\chi$.
\end{proof}

\begin{cor}
Since $\overline{\Psi(Y,X)}/\mathcal{K}$ is a GCR-$C^*$-algebra, by~\cite[Theorem 4.3.7]{Dix1} we have a bijection
$$
\widehat{\Sigma}\longrightarrow\Prim \Sigma,\quad \pi\in\widehat{\Sigma}\longmapsto\ker\pi\in\Prim\Sigma,
$$
where $\widehat{\Sigma}$ is the spectrum of $\Sigma$ (the set of equivalence classes of irreducible representations of $\Sigma$).
\end{cor}

\subsection*{$\Gamma$-Boutet de Monvel operators for pairs $(Y,X)$.}
We consider actions of $\Gamma$ on $C^*$-al\-gebras $\overline{\Psi(Y,X)}$ and $C_0(Y_\infty\sqcup X)^+$ by automorphisms and define the corresponding maximal $C^*$-crossed products (see e.g.~\cite{Ped1}), denoted by $\overline{\Psi(Y,X)}{\rtimes}\Gamma$ and $C_0(Y_\infty\sqcup X)^+	{\rtimes}\Gamma$, respectively.

Consider triples of the form $\mathcal{D}=(D, P_1, P_2)$ (cf.~\eqref{D_restr}), where
$$
D\in{\rm{Mat}}_N(\overline{\Psi(Y,X)}{\rtimes}\Gamma),\qquad
P_{1,2}\in{\rm{Mat}}_N(C_0(Y_\infty\sqcup X)^+	{\rtimes}\Gamma).
$$
Here $P_{1,2}$ are projections $(P_j^2=P_j)$ and we suppose that
\begin{equation}
\label{D_proj_G}
D=P_2DP_1.
\end{equation}
Then~\eqref{D_proj_G} implies that $D$ restricts to the $\Gamma$-operator
\begin{equation}
\label{D_restr_G}
D\colon P_1(\mathcal{H}\otimes\mathbb{C}^N)\longrightarrow P_2(\mathcal{H}\otimes\mathbb{C}^N),
\end{equation}
acting between the ranges of $P_1$ and $P_2$.

\subsection*{Ellipticity and the Fredholm property.}
\begin{definition}
A triple $(D, P_1, P_2)$ is {\em elliptic}, if there exists a triple $(R, P_2, P_1)$ such that
\begin{equation}
\label{ellip_rel}
\begin{aligned}
\sigma_{\rm{int}}(DR)&=\sigma_{\rm{int}}(P_2), &
\sigma_{\rm{int}}(RD)&=\sigma_{\rm{int}}(P_1),\\
\sigma_X(DR)&=\sigma_X(P_2), &
\sigma_X(RD)&=\sigma_X(P_1).
\end{aligned}
\end{equation}
\end{definition}

\begin{definition}
A triple $(D, P_1, P_2)$ is {\em trajectory elliptic}, if the two following conditions are satisfied:\\
1) for all $m\in S^*Y^+_\infty$ the interior trajectory symbol of $D$ is invertible as the operator in the spaces
\begin{equation}
\label{traj_ell_in_G}
\sigma_{\rm{int}}(D)(m):
P_1(m) l^2(\Gamma,\mathbb{C}^N) \longrightarrow
P_2(m) l^2(\Gamma,\mathbb{C}^N);
\end{equation}
2) for all $m'\in S^*X$ the boundary trajectory symbol of $D$ is invertible as the operator in the spaces
\begin{equation}
\label{traj_ell_b_G}
\sigma_X(D)(m'):
P_1(m') l^2(\Gamma,(L^2(\mathcal{N}^*_{m'}X)\oplus\mathbb{C})\otimes\mathbb{C}^N) \longrightarrow
P_2(m') l^2(\Gamma,(L^2(\mathcal{N}^*_{m'}X)\oplus\mathbb{C})\otimes\mathbb{C}^N).
\end{equation}
Here, for $h\in \Gamma$ and $\Theta\in \mathcal N^*_{m'}X$,
are the trajectory symbols of $P_j$, see~\eqref{eq-bnd-trj1}.
\end{definition}

The following result contains the statement in Theorem~\ref{tr_kon} as a special case.
\begin{theorem}
\label{tr_kon_G}
Suppose that $\Gamma$ is amenable and its action on the set of primitive ideals of the Calkin algebra $\overline{\Psi(Y,X)}/\mathcal{K}$, see~\eqref{symb_homeom_2}, is topologically free. Then the three following conditions are equivalent:\\
1) $(D, P_1, P_2)$ is elliptic;\\
2) $(D, P_1, P_2)$ is trajectory elliptic;\\
3) the operator \eqref{D_restr_G} associated with $(D, P_1, P_2)$ has the Fredholm property.
\end{theorem}
\begin{remark}
Condition 1) in Theorem~\ref{tr_kon_G} implies Conditions 2) and 3) without any assumptions on $\Gamma$ and its action.
\end{remark}
\begin{proof}
1. Let us prove that 1) implies 3). Indeed, given a triple $(R, P_2, P_1)$, consider the corresponding $\Gamma$-operator
$$
R: P_2(\mathcal{H}\otimes\mathbb{C}^N)\longrightarrow P_1(\mathcal{H}\otimes\mathbb{C}^N).
$$
We claim that $R$ is a two-sided almost inverse of~\eqref{D_restr_G}. Indeed,~\eqref{ellip_rel} implies
$$
\sigma_{\rm{int}}(DR-P_2)=0,\qquad \sigma_X(DR-P_2)=0.
$$
Therefore the element $DR-P_2\in{\rm{Mat}}_N(\overline{\Psi(Y,X)}{\rtimes}\Gamma)$ lies in the subalgebra ${\rm{Mat}}_N(\mathcal{K}{\rtimes}\Gamma)$ and the corresponding operator in $\mathcal{H}\otimes\mathbb{C}^N$ is compact. So we obtain
$$
DR=P_2+K_1,\qquad\text{ where }K_1\in\mathcal{K}.
$$
Similarly, one can prove that
$$
RD=P_1+K_2,\qquad\text{ where }K_2\in\mathcal{K}.
$$
Hence $R$ is an almost inverse operator for $D$ up to compact operators and the operator~\eqref{D_restr} has the Fredholm property.

2. Now we prove that 1) implies 2). Indeed, the inverse operators for the trajectory symbols~\eqref{traj_ell_in} and~\eqref{traj_ell_b} are the trajectory symbols of $R$:
\begin{align*}
\sigma_{\rm{int}}(R)(m):
P_2(m) l^2(\Gamma,\mathbb{C}^N) &\rightarrow
P_1(m) l^2(\Gamma,\mathbb{C}^N),
\\
\sigma_X(R)(m')\!:\!
P_2(m')
l^2(\Gamma,(L^2(\mathcal{N}^*_{m'}X) \oplus\mathbb{C}){\otimes}\mathbb{C}^N)
&\rightarrow
P_1(m')
l^2(\Gamma,(L^2(\mathcal{N}^*_{m'}X) \oplus\mathbb{C}){\otimes}\mathbb{C}^N).
\end{align*}

3. Now we prove that 3) implies 1). It is known that the operator
$D:{\rm Im}P_1\to{\rm Im}P_2$ has the Fredholm property if and only if the two operators
$$
D^*D:{\rm Im}P_1\to{\rm Im}P_1\quad\text{and}\quad DD^*:{\rm Im}P_2\to{\rm Im}P_2
$$
have the Fredholm property.
This condition is equivalent to the Fredholm property of the operators
\begin{equation}
\label{comp_adj}
D^*D+(1-P_1):
\mathcal{H}\otimes\mathbb{C}^N\longrightarrow
\mathcal{H}\otimes\mathbb{C}^N,\qquad
DD^*+(1-P_2):
\mathcal{H}\otimes\mathbb{C}^N\longrightarrow
\mathcal{H}\otimes\mathbb{C}^N.
\end{equation}
These operators are Fredholm if and only if the corresponding elements in the Calkin algebra ${\rm Mat}_N(\mathcal{B}(\mathcal{H})/\mathcal{K})$
are invertible. Denote by
$\mathcal{C}\subset{\rm Mat}_N(\mathcal{B}(\mathcal{H})/\mathcal{K})$
the $C^*$-subalgebra  generated by the shift operators $\widetilde{T}_\gamma$, $\gamma\in\Gamma$, and the operators in ${\rm{Mat}}_N(\overline{\Psi(Y,X)})$.

Consider the mapping
\begin{equation}
\label{subalg_mapp}
\begin{aligned}
\Sigma{\rtimes}\Gamma& \longrightarrow \mathcal{C},\\
a=\{a_{{\rm int},\gamma}, a_{X,\gamma}\}&\longmapsto
\Bigl[\sum\limits_{\gamma} \operatorname{Op} (a_{{\rm
int},\gamma},
a_{X,\gamma})\widetilde{T}_\gamma\Bigr],
\end{aligned}\quad
\text{ãäå } \Sigma=\overline{\Psi(Y,X)}/\mathcal{K},
\end{equation}
where ${\rm Op} (a_{\rm int}, a_X)$ stands for a Boutet de Monvel operator with an interior symbol $a_{\rm int}$ and a boundary symbol on $X$ equal to $a_X$.

The application of the isomorphism theorem~\cite[Theorem 12.17]{AnLe1} shows that the mapping~\eqref{subalg_mapp} is an isomorphism of $C^*$-algebras. The conditions in the isomorphism theorem are satisfied since by assumption $\Gamma$ is amenable and its action on the spectrum $\widehat{\Sigma}$ of the algebra of symbols is topologically free.

Using the isomorphism~\eqref{subalg_mapp} we obtain that the operators~\eqref{comp_adj} are Fredholm if and only if their symbols
$$
\sigma(D^*D+(1-P_1)),\quad \sigma(DD^*+(1-P_2))\in{\rm Mat}_N(\Sigma{\rtimes}\Gamma)
$$
are invertible.
These elements are invertible if and only if there exist elements $r_1,r_2$ such that $r_j=p_jr_jp_j$, where $p_j=\sigma(P_j), j=1,2$, and
\begin{equation}
\label{symb_comp_r}
\sigma(D^*D)r_1=r_1\sigma(D^*D)=p_1,\qquad\sigma(DD^*)r_2=r_2\sigma(D^*D)=p_2.
\end{equation}
It follows from~\eqref{symb_comp_r} that $(D,P_1,P_2)$ is elliptic, and as an almost inverse triple one can take a triple $(R,P_2,P_1)$, where $\sigma(R)=r_1\sigma(D^*)=\sigma(D^*)r_2$.

4. Let us now prove that 2) implies 1). Similar to the previous step of the proof the ellipticity of $(D,P_1,P_2)$ is equivalent to the invertibility of the elements
$$
\sigma(D^*D+(1-P_1)),\quad \sigma(DD^*+(1-P_2))\in{\rm Mat}_N(\Sigma{\rtimes}\Gamma).
$$
Also, the trajectory ellipticity of $(D,P_1,P_2)$ is equivalent to the invertibility of the trajectory symbols of the operators
$$
D^*D+(1-P_1),\quad DD^*+(1-P_2).
$$
We can apply~\cite[Theorem 12.17]{AnLe1}, which states that an element in a crossed product ${\rm Mat}_N(\Sigma{\rtimes}\Gamma)$ is invertible if and only if all its trajectory representations are invertible. This gives the equivalence of conditions 1) and 2) in Theorem~\ref{tr_kon_G}.
\end{proof}

\section{Generalization to Non-topologically Free Actions}\label{sec6}
Here we give a generalization of Theorem~\ref{tr_kon} and its proof in Section~\ref{sec_C*_theory} to the case when the group action is not topologically free. This is done using results in~\cite[Appendix~B]{AnLe2}.

We use the following notation:
\begin{itemize}
\item  $\Omega=\Prim\Sigma$ is the space of primitive ideals of the symbol algebra;
\item  $\pi_m:\Sigma\to\mathcal{B}(H_m), m\in\Omega$, are the corresponding irreducible representations;
\item
$\Omega_\gamma=\{m\in\Omega\mid\gamma m=m\}$ is the fixed point set of $\gamma$, while $\mathring{\Omega}_\gamma$ is its interior;
\item
$I_\gamma=\{a\in\Sigma\mid\pi_m(a)=0\,\;\;\forall m\notin \mathring{\Omega}_\gamma\}$ is an ideal in $\Sigma$;
\item
$\Gamma_m=\{\gamma\in\Gamma\mid m\in\mathring{\Omega}_\gamma\}$ is a subgroup in $\Gamma$.
\item
Moreover, we embed $\Sigma$ as a closed subalgebra of the bounded operators on the Hilbert space
\begin{equation}
\label{Sigma_H_sp}
\mathcal{H}=L^2(S^*Y)\oplus L^2(S^*X,L^2(\mathcal{N}^*X))
\end{equation}
with respect to suitable volume forms and
\item
we  define (similar to~\eqref{unit_repres}) the unitary representation
$\Gamma\to\mathcal{BH}$  by unitary shift operators denoted by $\overline{T}_\gamma$.
\end{itemize}
Consider the following conditions.
\begin{cond}\label{cond2}
$\Gamma_m$ is a normal subgroup in $\Gamma$ and $\Gamma^m=\Gamma/\Gamma_m$ the quotient group.
\end{cond}
\begin{cond}\label{cond3}
Given any $m\in\Omega$, $\Gamma^m$ {\it{acts topologically freely at}} $m$ in the following sense: for all finite sets $\{\gamma_1,\dots,\gamma_N\}\subset\Gamma\setminus\Gamma_m$ the union $\Omega_{\gamma_1}\cup\dots\cup\Omega_{\gamma_N}$ contains no neighborhood of $m$.
\end{cond}
\begin{cond}\label{cond4}
We have $\overline{T}_\gamma I_\gamma\subset I_\gamma$ for all $\gamma\in\Gamma$ and moreover there exists a multiplier $u_\gamma\in M(I_\gamma)$ such that $\overline{T}_\gamma a=u_\gamma a,$ $\forall a\in I_\gamma$.
\end{cond}
Conditions \ref{cond2}, \ref{cond3}, and \ref{cond4} are precisely the assumptions of the isomorphism theorem in~\cite[Theorem B.28]{AnLe2} and if they are satisfied, then this theorem implies that there is an injection of $C^*$-algebras
\begin{equation}
\label{C^*_inj}
\overline{\Psi_\Gamma(Y,X)}/\mathcal{K}\stackrel{\bigoplus\limits_m{p_m}}{\longrightarrow}\bigoplus\limits_{m\in\Omega}{\mathcal{B}(l^2(\Gamma^m,H_m))},
\end{equation}
where $\Psi_\Gamma(Y,X)\subset\mathcal{B}(L^2(Y)\oplus L^2(X))$ is the algebra generated by $\Psi(Y,X)$ and the shift operators $\widetilde{T}_\gamma$. The mapping $p_m$ in~\eqref{C^*_inj} is called the {\it{trajectory symbol at the point}} $m$ and is defined as follows:
\begin{itemize}
\item we choose an arbitrary mapping $\delta:\Gamma^m\longrightarrow\Gamma$ such that $\delta_e=e$; $\delta_{\gamma^{-1}}=(\delta_\gamma)^{-1}$ and $\pi\delta_\gamma=[\gamma]$, where $\pi:\Gamma\longrightarrow\Gamma^m$ is the natural projection;
\item given $a\in\Sigma,$ we set
\begin{equation}
\label{traj_s_1}
    [p_m(a)f](h)=\pi_m(\delta_h(a))f(h);
\end{equation}
\item given $\gamma \in\Gamma_m$, we set
\begin{equation}
\label{traj_s_2}
 [p_m(\widetilde{T}_\gamma)f](h) =\pi_m(\delta_h(u_\gamma))f(h),
\end{equation}
where the element $\pi_m(\delta_h(u_\gamma))\in  \mathcal{B} (H_m) $ is uniquely defined in terms of the multiplier $u_{\gamma }$ above (see~\cite[pages~327-328]{AnLe2} for more details);
\item  given $\gamma\in\Gamma$, we set
\begin{equation}
\label{traj_s_3}
[p_m(\widetilde{T}_{\delta_\gamma})f](h)=\pi_m(\delta_h(u_{\delta_\gamma\delta^{-1}_{h\gamma}\delta_{ h}}))f(h\gamma);
\end{equation}
\item now, given an arbitrary sum
\begin{equation}
\label{arb_sum}
A=\sum\limits_{\gamma\in\Gamma}A_\gamma\widetilde{T}_\gamma\in\Psi_\Gamma(Y,X),
\end{equation}
we set
\begin{equation}
\label{p_m[A]}
p_m[A]=\sum\limits_{\gamma\in\Gamma}p_m(\sigma(A_\gamma))p_m(\widetilde{T}_{\gamma\delta_\gamma^{-1}})p_m(\widetilde{T}_{\delta_\gamma})\in\mathcal{B}(l^2(\Gamma^m,H_m)),
\end{equation}
where we use~\eqref{traj_s_1},~\eqref{traj_s_2},~\eqref{traj_s_3} to define the factors.
\end{itemize}

It follows from~\cite[Theorem B.28]{AnLe2} that $p_m[A]$ is well defined (i.e., it does not depend on the choice of the representative~\eqref{p_m[A]} in the class $[A]$) and extends to a representation of the $C^*$-algebra $\overline{\Psi_\Gamma(Y,X)}/\mathcal{K}$. Moreover, the direct sum of all such representations over all $m\in\Omega$, see~\eqref{C^*_inj}, is injective and $[A]$ is invertible if and only if all $p_m[A]$ are invertible.

\subsection*{Example}
Our setting is as follows:
\begin{itemize}
\item We have an ambient manifold $Y=\mathbb{S}^1\times \mathbb{R}$ and a codimension zero submanifold with boundary $M=\mathbb{S}^1\times[0,1]$;
\item $\mathbb{Z}$ acts on $Y$: $(x,t)\longmapsto(x,2^kt)$, $k\in\mathbb{Z}$;
\item $X=\mathbb{Z}(\partial M)=\mathbb{S}^1\times\left(
\{0\}\cup\{2^k,k\in\mathbb{Z}\}
\right)$ with the topology of the disjoint union;
\item $X_N=\mathbb{S}^1\times(\{0\}\cup\{2^k:|k|\leq N\})$;
\item $Y_\infty=\mathbb{S}^1\times \left(
(-\infty,0]\sqcup \bigsqcup\limits_k [2^k,2^{k+1}]
\right)$ is obtained from $Y$ by an infinite number of cuts.
\end{itemize}
We consider the following spaces and operators
\begin{itemize}
\item the space $\mathcal{H}=L^2(Y)\oplus L^2(X)$ and the unitary shift operator
\begin{equation*}
\begin{split}
\widetilde{T}\colon\mathcal{H}&\longrightarrow \mathcal{H};\\
(u(x,t),v(x,t'))&\xrightarrow{\widetilde{T}}
(2^{-1/2}u(x,2^{-1}t),v(x,2^{-1}t')),\quad
t'\in\{0\}\cup\{2^k,k\in\mathbb{Z}\};
\end{split}
\end{equation*}
\item the algebra $\Psi(Y,X)=\bigcup\limits_{N\geq1}\Psi_X(Y,X_N)$, where $\Psi_X(Y,X_N)$ is the algebra of transmission problems on $Y$  of order and type zero with transmission conditions on $X_N$ and we suppose in addition that $\psi$DOs have the transmission property on $X$;
\item the $\mathbb{Z}$-operators on $Y$ of the form
\begin{equation}
\label{z_oper}
A=\sum\limits_k A_k\widetilde{T}^k:\mathcal{H}\longrightarrow\mathcal{H},\text{ where }A_k\in\Psi(Y,X).
\end{equation}
\end{itemize}

\subsection*{\bf{Trajectory symbols of $\mathbb{Z}$-operators.}}
For $\mathbb{Z}$-operators of the form \eqref{z_oper} we define trajectory symbols at different points:

1. Given $m=(x,t,\xi,\tau)\in S^\vee Y^+_\infty\subset\Omega$, the action of $\mathbb{Z}$ on $\Omega=\mathfrak{Y}_X=S^\vee Y^+_\infty\cup S^*X$ is topologically free at $m$ and the interior trajectory symbol
\begin{equation}
\label{int_symb}
\sigma_{\rm{int}}(A):l^2(\mathbb{Z})\longrightarrow l^2(\mathbb{Z}),\quad(\sigma_{\rm{int}}(A)(m)u)(n)=\sum\limits_k\sigma_{\rm{int}}(A_k)
(x,2^{-n}t,\xi,2^n\tau)u(n+k)
\end{equation}
is a difference operator with variable coefficients;

2. Given $m'=(x,\xi,t)\in S^*X\setminus \{S^*(\mathbb{S}^1\times\{0\})\}\subset\Omega$, the action of $\mathbb{Z}$ on $\Omega$ is topologically free at $m'$ and the boundary trajectory symbol
\begin{equation}
\label{boun_symb_1}
\sigma_X(A)(m'):l^2(\mathbb{Z},L^2(\mathbb{R})\oplus\mathbb{C})
\longrightarrow
l^2(\mathbb{Z},L^2(\mathbb{R})\oplus\mathbb{C})
\end{equation}
$$
(\sigma_X(A)(m')u)(n)=\sum\limits_k
{\kappa}_{2^n}\sigma_X(A_k)
(x,\xi,2^{-n}t){\kappa}^{-1}_{2^n}u(n+k)
$$
is a difference operator with variable operator coefficients. Note that the operators ${\kappa}_\lambda$ were defined in~\eqref{kappalambda}.

3. Given $m'=(x,\xi)\in S^*(\mathbb{S}^1\times\{0\})\subset\Omega$, we claim that the action of $\mathbb{Z}$ on $\Omega$ is not topologically free at $m'$. Indeed, given $k\ne 0\in \mathbb{Z}$, its fixed point set is an open set equal to
$$
\Omega_k=\mathring{\Omega}_k=S^*(\mathbb{S}^1\times\{0\})  .
$$
Hence, the action is not topologically free at $m'$. Let us check that Conditions~\ref{cond2},\ref{cond3},\ref{cond4} are satisfied in this case. We take $k\ne 0$ and consider the ideal
$$
 I_k=C(S^*(\mathbb{S}^1\times\{0\}),\mathcal{K}(L^2(\mathbb{R})\oplus\mathbb{C})),
$$
which is clearly independent of $k$. Then, we consider groups $\Gamma_{m'}=\mathbb{Z}$, $\Gamma^{m'}=\{0\}$. This implies that Conditions~\ref{cond2},\ref{cond3} are fulfilled.   Finally, we define the multiplier in Condition~\ref{cond4}
$$
u={\kappa}_{2}\in M(I_k)\quad\text{such that } \overline{T}^ka=u^k a,\quad \forall k\ne 0, a\in I_k.
$$
Then we obtain the boundary symbol (we set $\delta_0=0$ in \eqref{traj_s_1} and \eqref{traj_s_2})
\begin{equation}
\label{boun_symb_2}
\begin{aligned}
\sigma_X(A)(m')\colon L^2(\mathbb{R})\oplus\mathbb{C}&\longrightarrow
L^2(\mathbb{R})\oplus\mathbb{C},
\\
\sigma_X(A)(m')u&=\sum\limits_k
\sigma_X(A_k)
(x,\xi,t=0){\kappa}_{2^k}u,
\end{aligned}
\end{equation}
as a $\psi$DO on $\mathbb{R}$ with dilations.

Thus,  we see that conditions of the isomorphism theorem~\cite[Theorem B.28]{AnLe2} are satisfied and we get the following corollary.
\begin{cor}
The operator $A$ has the Fredholm property if and only if all trajectory symbols~\eqref{int_symb}, \eqref{boun_symb_1}, \eqref{boun_symb_2} are invertible.
\end{cor}

\subsection*{\bf{Restriction to the submanifold with boundary.}}

Let us now obtain the Fredholm criterion for the restrictions of our operators to $M=\mathbb{S}^1\times[0,1]\subset Y=\mathbb{S}^1\times\mathbb{R}$. Consider the projection
$$
\mathcal{P}=(\chi_M,\chi_{\partial M})\in
C^\infty_c(Y_1)\oplus C^\infty_c(X)^+\subset \Psi(Y,X).
$$
Its symbols are equal to
\begin{align*}
\sigma_{\rm{int}}(\mathcal{P})(x,t,\xi,\tau)&=\chi_M(x,t)\in
C(S^\vee Y^+_\infty),
\\
\sigma_X(\mathcal{P})(x,t,\xi):
L^2(\mathbb{R})\oplus\mathbb{C} &\longrightarrow
L^2(\mathbb{R})\oplus\mathbb{C}, \text{ ãäå }
t\in\{0\}\cup\{2^k,k\in\mathbb{Z}\} ,
\\
\sigma_X(\mathcal{P})(x,t,\xi)&=
\left\{
\begin{aligned}
\chi_M(x,t)\oplus \operatorname{Id}, & \quad\text{ åñëè } t\ne0,
 1,\\
\Pi^+\oplus\operatorname{Id}, & \quad \text{ åñëè } t=0,\\
\Pi^-\oplus\operatorname{Id}, & \quad\text{ åñëè } t=1.\\
\end{aligned}
\right.
\end{align*}
Given a $\mathbb{Z}$-operator  \eqref{z_oper}, we consider its restriction to $M=\mathbb{S}^1\times[0,1]$
\begin{equation}
\label{Z_op_restr}
\mathcal{P}A\mathcal{P}:\mathcal{PH} \longrightarrow\mathcal{PH},\;\;\text{where }
\mathcal{PH}=L^2(\mathbb{S}^1\times[0,1])\oplus
\bigoplus\limits_{t\in\{0,2^k,k\leq0\}} L^2(\mathbb{S}^1\times\{t\})
\end{equation}
to $M=\mathbb{S}^1\times[0,1]$. The trajectory symbols of the operator~\eqref{Z_op_restr} are described as follows.

Interior trajectory symbol at $m=(x,t,\xi,\tau)$, $\frac12<t\leq 1$, is equal to
\begin{equation*}
\begin{split}
\sigma_{\rm{int}}(\mathcal{P}A\mathcal{P})(m)&\colon
l^2(\mathbb{Z}_+)\longrightarrow l^2(\mathbb{Z}_+),\\
(\sigma_{\rm{int}}(\mathcal{P}A\mathcal{P}u))(n)
&=\sum\limits_{k\geq-n}\sigma_{\rm{int}}
(A_k)(x,2^{-n}t,\xi,2^n\tau)u(n+k).
\end{split}
\end{equation*}
Interior trajectory symbol at $m=(x,t,\xi,\tau)$, $t=0$, is equal to
$$
\sigma_{\rm{int}}(\mathcal{P}A\mathcal{P})(m):l^2(\mathbb{Z})\longrightarrow l^2(\mathbb{Z}),\quad
(\sigma_{\rm{int}}(\mathcal{P}A\mathcal{P}u))(n)=\sum\limits_{k}\sigma_{\rm{int}}(A_k)(x,0,\xi,2^n\tau)u(n+k).
$$
Boundary trajectory symbol at $m'=(x,t,\xi), t=1,$ is equal to
$$
\sigma_X(A)(m'): l^2(\mathbb{Z}_{\geq 1},L^2(\mathbb{R})\oplus\mathbb{C})\oplus \overline{H^-} \longrightarrow
l^2(\mathbb{Z}_{\geq 1},L^2(\mathbb{R})\oplus\mathbb{C})\oplus\overline{H^-}
$$
$$
(u(n),u(0))\longmapsto
\left\{
\begin{matrix}
&\sum\limits_{k\geq-n}{\kappa}_{2^n}\sigma_X(A_k)(x,2^{-n},\xi)
{\kappa}^{-1}_{2^n}u(n+k), & \text{if }n\geq 1 \\
&\Pi^-(\sum\limits_{k\geq-n}{\kappa}_{2^n}\sigma_X(A_k)(x,2^{-n},\xi)
{\kappa}^{-1}_{2^n}u(n+k)), & \text{if }n=0.
\end{matrix}
\right.
$$
Boundary trajectory symbol at $m'=(x,0,\xi)$ is equal to
$$
\sigma_X(A)(m'): \overline{H^+}\oplus\mathbb{C} \longrightarrow
\overline{H^+}\oplus\mathbb{C},\qquad
u\longmapsto \Pi^+
\sum\limits_k\sigma_X(A_k)(x,0,\xi)
{\kappa}_{2^k}u.
$$

\section{Appendix. Some Results from $C^*$-algebras}
Let $A$ be a $C^*$-algebra and $\widehat{A}$  its spectrum, i.e., set of the equivalence classes of irreducible representations $\pi:A\to\mathcal{B}(H)$, where $H$ is a Hilbert space. We shall denote by $\Prim A$  the set of primitive ideals, i.e., $\Prim A=\{\ker\pi\subset A| \pi\in \widehat{A}\}$.

\begin{itemize}
\item \cite[Definition 4.2.1]{Dix1}: $A$ is a CCR-algebra (or liminal algebra) $\Leftrightarrow \forall\pi\in\widehat{A}, a\in A$ the operator $\pi(a)$ is compact.
\item \cite[Corollary 10.4.5]{Dix1}: Given a locally compact space $X$, $C_0(X,\mathcal{K})$ is a CCR-al\-ge\-bra and one has a canonical homeomorphism $\widehat{C_0(X,\mathcal{K})}\simeq X.$
\item \cite[Definition 4.3.1]{Dix1}: $A$ is a GCR-algebra (or postliminal algebra) $\Leftrightarrow
\forall$ ideals $J\ne A$ the quotient $A/J$ contains a nonzero CCR ideal.
\item \cite[Proposition 4.3.4]{Dix1}: Let $A$ be a $C^*$-algebra. Then $A$ is a GCR-algebra $\Leftrightarrow$ $\exists$ a composition series
$0=I_0\subset I_1\subset \dots \subset I_\alpha=A$, i.e. an increasing family of ideals, where $I_{p+1}/I_p$ is a CCR-algebra for all $p$.
\item \cite[Theorem 4.3.7]{Dix1}: If $A$ is a GCR-algebra, then the mapping $\widehat{A}\to\Prim{A}, \pi\mapsto\ker\pi$ is a bijection.
\item \cite[Definition 3.1.1]{Dix1}: The Jacobson topology on $\Prim A$: for a set $T\subset \Prim A$ we define
$$
 \overline{T}=\Bigl\{J\in\Prim A |\bigcap\limits_{I\in T}{I\subset J}\Bigr\}.
$$
This closure operation uniquely generates a topology on $\Prim A$.
\item \cite[Theorems 3.4.10 and 3.4.4]{Dix1}: Let $\pi\in\widehat{A}, S\subset \widehat{A}$. Then $\pi\in\overline{S}\Leftrightarrow\pi$ is weakly contained in $S\Leftrightarrow\bigcap\limits_{\rho\in S}\ker\rho\subset\ker\pi$.
\item \cite[Propositions 2.11.5 and 3.2.1]{Dix1}: Let $I\subset A$ be an ideal. Let
$$
\Prim_I(A)=\{J\in\Prim A| I\subset J\},
$$
$$
\Prim^I(A)=\{J\in\Prim A| I\not\subset J\}.
$$
Then we have the homeomorphisms
$$
\Prim_I(A)\longrightarrow\Prim(A/I)\qquad J\longmapsto J/I,
$$
$$
\Prim^I(A)\longrightarrow\Prim(I)\qquad J\longmapsto J\cap I,
$$
where the set $\Prim_I(A)$ is closed in $\Prim A$ and $\Prim^I(A)$ is open in $\Prim A$.
\end{itemize}

\end{document}